\documentclass[11pt]{amsart}
\usepackage{amssymb,amsmath,amsthm}
\usepackage[initials]{amsrefs}
\usepackage{bm}
\usepackage{fullpage}
\usepackage{enumerate}
\usepackage{verbatim}
\usepackage{cite}

\usepackage[hidelinks]{hyperref}

\DeclareFontFamily{U}{mathx}{\hyphenchar\font45}
\DeclareFontShape{U}{mathx}{m}{n}{
      <5> <6> <7> <8> <9> <10>
      <10.95> <12> <14.4> <17.28> <20.74> <24.88>
      mathx10
      }{}
\DeclareSymbolFont{mathx}{U}{mathx}{m}{n}
\DeclareFontSubstitution{U}{mathx}{m}{n}
\DeclareMathAccent{\widecheck}{0}{mathx}{"71}

\newtheorem{theorem}[equation]{Theorem}
\newtheorem{lemma}[equation]{Lemma}
\newtheorem{prop}[equation]{Proposition}
\newtheorem{cor}[equation]{Corollary}

\newtheorem{remark}[equation]{Remark}

\numberwithin{equation}{section}

\newcommand{\R}{\mathbb{R}}

\newcommand{\Z}{\mathbb{Z}}

\newcommand{\Sph}{\mathbb{S}}

\newcommand{\Met}{\operatorname{Met}}

\newcommand{\madm}{m_{_{ADM}}}
\newcommand{\mh}{m_{_H}}
\newcommand{\mby}{m_{_{BY}}}
\newcommand{\mb}{m_{_B}}
\newcommand{\ms}{m_{_S}}

\newcommand{\Scal}{\mathcal{S}}
\newcommand{\Bcal}{\mathcal{B}}
\newcommand{\Hcal}{\mathcal{H}}

\newcommand{\Sdot}{\dot{\Scal}}
\newcommand{\Bdot}{\dot{\Bcal}}

\newcommand{\Ric}{\operatorname{Ric}}

\renewcommand{\div}{\operatorname{div}}
\newcommand{\tr}{\operatorname{tr}}

\newcommand{\up}[2]{\overset{\scriptscriptstyle #2}{#1}\vphantom{#1}}
\newcommand{\upp}[2]{\overset{\scriptscriptstyle (#2)}{#1}\vphantom{#1}}
\newcommand{\trl}[1]{\underset{\circ}{#1}\vphantom{#1}}

\newcommand{\abs}[1]{\left\lvert#1\right\rvert}
\newcommand{\norm}[1]{\left\|#1\right\|}

\title{Second-order mass estimates
       for static vacuum metrics
       with small Bartnik data}

\author{David Wiygul}
\address{
  Università degli studi di Trento
  Via Sommarive, 14 - 38123 Povo TN, Italy
}
\email{davidjames.wiygul@unitn.it}

\begin{document}

\begin{abstract}
Given on the $2$-sphere Bartnik data
(prescribed metric and mean curvature)
that is a small perturbation
of the corresponding data for the standard unit sphere in Euclidean space,
we estimate to second order,
in the size of the perturbation,
the mass of the asymptotically flat static vacuum extension
(unique up to diffeomorphism)
which is a small perturbation
of the flat metric
on the exterior of the unit ball in Euclidean space
and induces the prescribed data on the boundary sphere.
As an application we obtain
a new upper bound
on the Bartnik mass of small metric spheres
to fifth order in the radius.
\end{abstract}

\maketitle

\section{Motivation and statement of the results}
\label{intro}
Bartnik's definition of quasilocal mass
(\cites{BarNew, BarICM, BarTHL})
has inspired a number of interesting
questions and results in mathematical relativity;
see for example the surveys
\cites{AndSurvey, CPC}.
The present article
is motivated by a desire
to refine existing estimates
of the Bartnik mass of $2$-spheres
with small data,
meaning metric and mean curvature
close to those of the standard sphere.
Over time Bartnik's original definition
has given rise in the literature to several variants,
whose relation has been analyzed 
in \cites{Jdef, McDef},
but here
we adopt a somewhat restrictive version,
defined
in \eqref{mbdef},
after establishing some notation
and preliminary definitions.
The statement of our results begins on page \pageref{mainthm}.

\subsection*{Basic definitions and notation}
Set
$M:=\{\vec{x} \in \R^3 \; : \; \abs{\vec{x}} \geq 1\}$
and let
$\iota: \partial M \to M$
be the inclusion map for its boundary $\Sph^2:=\partial M$.
Given a Riemannian metric $g=g_{ab}$ on $M$,
we write $R_{abcd}[g]=R_{abcd}$,
$R_{ab}[g]=R_{ab}$,
and $R[g]=R$
for the corresponding Riemann, Ricci, and scalar curvature
of $g$; 
our conventions for $R_{abcd}$
are declared in Appendix \ref{Riem}.
We also write $\Hcal[\iota,g]$ for
the corresponding mean curvature function induced on $\partial M$:
by our convention
the $g$-divergence of the outward
(so pointing into the unit ball)
$g$-unit normal on $\partial M$.
The volume and area densities
induced by $g$ and $\iota^*g$ respectively
will be denoted by $\sqrt{\abs{g}}$
and $\sqrt{\abs{\iota^*g}}$.
We reserve $\delta=\delta_{ab}$ for the Euclidean metric on $\R^3$
(and its restriction to $M$)
and caution that by our convention $\Hcal[\iota,\delta]=-2$.
In integrals of functions on $M$ and $\Sph^2$
the densities $\sqrt{\abs{\delta}}$ and $\sqrt{\abs{\iota^*\delta}}$
respectively
should be understood,
unless a different density is explicitly specified.

In tensor expressions we will use Roman indices
for components in $M$
and Greek indices for components in $\partial M$.
Mostly we employ abstract index notation
(that is coordinate-free tensor notation),
but we make two types of occasional exceptions.
First, we will sometimes work
in the standard Cartesian coordinates
on $\R^3$,
for which purpose we reserve
the indices $i,j,k$;
second, the label $r$ will always
refer to the standard radial coordinate
on $\R^3$.
Roman indices will be raised and lowered
via $\delta$ and Greek via $\iota^*\delta$.
We shall use a semicolon (in the case of fields over $M$)
or colon (in the case of fields over $\partial M$)
to indicate differentiation
with respect to the Levi-Civita connection
induced by $\delta$ or $\iota^*\delta$ respectively;
in the case of covariant differentiation
of functions
(which does not depend on the choice of metric)
and in the case of coordinate differentiation
(with respect to a Cartesian or radial coordinate)
we prefer to use a comma instead.

Given $k \in \Z \cap [0,\infty)$,
$\alpha \in [0,1)$,
$\beta \in \R$,
and a $C^k_{loc}$ tensor field $F$ over a manifold $S$,
possibly with boundary, smoothly embedded in $\R^3$
(for example $S=M$ or $S=\partial M$),
we define the standard
H\"{o}lder norm
\begin{equation}
  \norm{F}_{k,\alpha}
  :=
  \sum_{i=0}^k \sup_{\vec{x} \in S} 
        \abs{D_\delta^iF(\vec{x})}_{\delta}
    +\sup_{\vec{x} \neq \vec{y} \in S} 
      \frac{\abs{D_\delta^kF(\vec{x})- D_\delta^kF(\vec{y})}_\delta}
        {\abs{\vec{x}-\vec{y}}^\alpha}
\end{equation}
and weighted H\"{o}lder norm
  \begin{equation}
  \begin{aligned}
    \norm{F}_{k,\alpha,\beta} 
    :=
    &\sum_{i=0}^k \sup_{\vec{x} \in S} 
        (1+\abs{\vec{x}})^{\beta+i}\abs{D_\delta^iF(\vec{x})}_{\delta} \\
    &+\sup_{\vec{x} \neq \vec{y} \in S} 
      \left[
        1+\min \left\{\abs{\vec{x}},\abs{\vec{y}}\right\}
      \right]^{\alpha+\beta+k}
        \frac{\abs{D_\delta^kF(\vec{x})- D_\delta^kF(\vec{y})}_\delta}
        {\abs{\vec{x}-\vec{y}}^\alpha},
  \end{aligned}
  \end{equation}
where the derivatives $D_\delta$ and differences
are taken componentwise relative
to the standard Cartesian coordinates $\{x^1,x^2,x^3\}$ on $\R^3$.
For each tensor bundle $E$ over $S$
we define the Banach spaces
$C^{k,\alpha,\beta}(E)$ and $C^{k,\alpha}(E)$
(written simply $C^{k,\alpha,\beta}(S)$ and $C^{k,\alpha}(S)$ as usual
when $E$ is the trivial bundle $S \times \R$)
of sections of $E$ with finite $\norm{\cdot}_{k,\alpha,\beta}$ 
and $\norm{\cdot}_{k,\alpha}$ norms respectively.
(Frequently the letters $\alpha$ and $\beta$
will also appear as indices for tensors over $\partial M$,
but this double duty should never cause confusion.)

We will make routine use of the standard
asymptotic ``big O'' notation.
Specifically,
suppose $D$ (for data) is a set, $\mathcal{X}$ is a vector space equipped with norm $\norm{\cdot}$,
and $x,y: D \to \mathcal{X}$ and $c: D \to [0,\infty)$ are functions;
we write
  \begin{equation}
    x=y+O(c)
  \end{equation}
if there exist constants $C, \epsilon_0 > 0$ such that
  \begin{equation}
    \norm{x(d)-y(d)} \leq Cc(d) \mbox{ for all } d \in D
      \mbox{ with } c(d)<\epsilon_0.
  \end{equation}
In instances of this notation the set $D$, vector space $\mathcal{X}$,
and norm $\norm{\cdot}$ should always be clear from context.
Frequently $\mathcal{X}$ will simply be $\R$
with $\norm{\cdot}$ the absolute value,
and typically $D$ will consist either of Bartnik data
(prescribed metric and mean curvature on a given surface)
close to that of the standard sphere
or of metric balls with small radius
and fixed center in a given Riemannian manifold.

We write 
$\Met^{k,\alpha,\beta}(M)$
for the space of Riemannian metrics on $M$
whose difference from the Euclidean metric $\delta$
belongs to
$C^{k,\alpha,\beta}\left(T^*M^{\odot 2}\right)$.
Additionally we call $\partial M$
outer-minimizing in $(M,g)$
if its area (measured by $g$)
is no greater than that of any surface in $M$ enclosing it.
Note that $\partial M$ is outer-minimizing
in $(M,g)$ for every $g$ in a sufficiently small
$C^{1}$ neighorhood of $\delta$.

Now
to a given $C^2$ Riemannian metric
$\gamma=\gamma_{\alpha\beta}$
and $C^1$ function $H$
on $\Sph^2=\partial M$ 
we associate the Bartnik mass
  \begin{equation}
  \label{mbdef}
    \begin{aligned}
      \mb[\gamma,H]
      :=
      \inf 
      \{
      &\madm[g] \; : \;
        g \in \Met^{2,0,1}(M), \\
        &R[g] \geq 0, \; \iota^*g=\gamma, \;
         \Hcal[\iota,g] \geq H, \mbox{ and} \\ 
        &\partial M \mbox{ is outer-minimizing in } (M,g)
      \}
    \end{aligned}
  \end{equation}
(provided this infimum exists), where
\begin{equation}
\label{madmdef}
    \madm[g]
    :=
    \frac{1}{16\pi} \lim_{r \to \infty}
     \int_{\abs{x}=r}
       \left(
         g_{ik}^{\;\;\; ,k}-g^k_{\;\; k,i}
       \right)
       \frac{x^i}{\abs{x}}
  \end{equation}
is the ADM mass
(\cites{ADM, BarADM, Chr, OM})
of $(M,g)$,
commas indicating coordinate differentiation as usual.
The inequality on mean curvature
was adopted from \cites{MiaoVarEff}
and enforces
nonnegative scalar curvature across the boundary
in a distributional sense;
we remind the reader that by our conventions
the standard unit sphere in Euclidean space
has mean curvature $\Hcal[\iota, \delta]=-2$.
The outer-minimizing condition
was suggested in
\cites{Bray}.
Note that,
by 
the connectedness
(as in \cites{Munkres} or \cites{Smale})
of the space of orientation-preserving
diffeomorphisms of $\Sph^2$,
definition \eqref{mbdef} is diffeomorphism-invariant:
$\mb[\gamma, H]=\mb[\phi^*\gamma, \phi^*H]$
for any smooth diffeomorphism $\phi: \Sph^2 \to \Sph^2$.

In turn we also define the Bartnik mass
for compact Riemannian $3$-balls:
if $(\Omega, h)$ is the image
of a smooth diffeomorphism $\phi$
from origin-centered closed unit ball in $\R^3$
and if $h$ induces on $\partial \Omega$
metric $\gamma$ and mean curvature $H$
(the $h$ divergence of the inward unit normal),
then
to $(\Omega,h)$ we associate the Bartnik mass
  \begin{equation}
  \label{mbomdef}
    \mb[\Omega,h]
    :=
    \mb[\phi^*\gamma, \phi^*H].
  \end{equation}
Note that the definition
does not depend on the choice of diffeomorphism.
By 
the positive mass theorem with corners
of \cites{MiaoCorners}
(as well as the extendibility of diffeomorphisms
 between balls established in \cites{Palais})
$\mb[\Omega,h]$ exists and is nonnegative
whenever $(\Omega, h)$
admits at least one embedding
into an asymptotically flat $\R^3$ having nonnegative scalar curvature
and in which $\partial \Omega$ is outer-minimizing.
Moreover it follows from \cites{HI}
that in this case
$\mb[\Omega, h]=0$
only when $h$ is flat.

\subsection*{Quasilocal mass of small spheres}
When $\Omega$ lies in a time-symmetric slice
whose source fields contribute nonnegative energy density at each point,
definition \eqref{mbomdef} is intended
as a measure of the (quasilocally defined) mass of $\Omega$.
In this article we specialize to the case where
$\Omega$ is a metric ball $B_\tau$ of small radius $\tau$
in some such slice $(N,h)$.
Fixing $(N,h)$ and the center of $B_\tau$
while taking $\tau$ small,
the leading contribution to the mass of $B_\tau$,
according to any physically reasonable definition,
must be given by its volume times the source 
energy density at its center.
Indeed, as reviewed in a moment,
the Bartnik mass does not fail to meet this natural expectation,
and
the present article is motivated by a desire to identify
the next most significant contributions.
See for example
\cites{BrownLauYork, Chen, ChenWangYau, FanShiTam, HorSch, Sza, WangJ, Yu}
for estimates of other quasilocal masses of small spheres
in both Riemannian and spacetime settings.

For context and further motivation
we now recall analogous estimates
for two quasilocal masses which provide
well-known lower and upper bounds
for the Bartnik mass
(at least under the assumptions of interest here).
Namely, for a surface $\Sigma$ having induced metric $\gamma$
and mean curvature $H$ in some time-symmetric slice
we have the Hawking mass (\cites{HMass})
$\mh[\Sigma]$
and the Brown-York mass (\cites{BY})
$\mby[\Sigma]$
of $\Sigma$:
\begin{equation}
  \begin{aligned}
    \mh[\Sigma]
      &:=
      \sqrt{\frac{\abs{\Sigma}}{16\pi}}
        \left(1-\frac{1}{16\pi}
          \int_{\Sigma} H^2 \, \sqrt{\abs{\gamma}}\right)
      \mbox{ and} \\
    \mby[\Sigma]
      &:=
      \frac{1}{8\pi} \int_{\Sigma} (H-H_0) \, \sqrt{\abs{\gamma}},
  \end{aligned}
\end{equation}
where $\abs{\Sigma}$ is the area of $\Sigma$
and for $\mby[\Sigma]$ we assume that $\Sigma$
has positive Gaussian curvature and that
$H_0$ is the mean curvature
of an isometric embedding of $\Sigma$ in Euclidean $\R^3$
(whose existence and uniqueness up to rigid motions
 are guaranteed by
 \cites{Nir, Pog, Herglotz, Sacksteder});
we also remind the reader that $\sqrt{\abs{\gamma}}$
denotes the area density induced by $\gamma$
and that by our conventions
the standard unit sphere has mean curvature $-2$.

We now fix a smooth
Riemannian $3$-manifold $(N,h)$ having nonnegative scalar curvature,
we also fix a point $p \in N$,
and for small variable $\tau > 0$
we consider the closed metric ball $B_\tau$ of center $p$ and radius $\tau$.
It is straightforward to compute the expansion
\begin{equation}
\label{hexp}
  \mh[B_\tau]
    :=
    \mh[\partial B_\tau]
    =
    \frac{\tau^3}{12}R+\frac{\tau^5}{720}
      \left(6\Delta R - 5R^2\right) + O(\tau^6),
\end{equation}
where $R$ is the scalar curavture, $\Delta$ the Laplacian,
and all terms are evaluated at $p$;
see for example \cites{FanShiTam}.
Much less straightforward but also accomplished in
\cites{FanShiTam}
is the corresponding calculation
\begin{equation}
\label{byexp}
  \mby[B_\tau]
    :=
    \mby[\partial B_\tau]
    =
    \frac{\tau^3}{12}R+\frac{\tau^5}{1440}
      \left(24\abs{\Ric}^2-13R^2+12\Delta R\right)
      +O(\tau^6),
\end{equation}
where $\abs{\Ric}^2$ is the squared norm of the Ricci curvature
(and again all terms are evaluated at $p$).

With \eqref{mbdef} as our definition of Bartnik mass
(and also under more general conditions)
we have the well-known inequalities
\begin{equation}
\label{sandwich}
  \mh[B_\tau] \leq \mb[B_\tau] \leq \mby[B_\tau],
\end{equation}
where the lower bound follows from \cites{HI}
and the upper bound from \cites{ShiTam}
(at least for $\tau$ small enough that
 $\partial B_\tau$ has positive Gaussian curvature
 and inward pointing mean curvature).
Since the upper and lower bounds evidently agree to fourth order,
we immediately obtain the estimate
\begin{equation}
\label{leading}
  \mb[B_\tau]
  =
  \frac{1}{12}R\tau^3 + O(\tau^5),
\end{equation}
confirming (via the energy constraint)
the expectation mentioned above
that as $\tau$ shrinks to $0$
the ratio of the ball's Bartnik mass to its volume
tends to the source energy density at its center.
On the other hand the fifth-order terms for the
Hawking and Brown-York mass differ in general
(even for spherically symmetric $B_\tau$).

\subsection*{Results}
Bartnik conjectured (\cites{BarNew, BarTHL, BarICM})
that his quasilocal mass should be realized by a unique
asymptotically flat static vacuum extension
inducing the given boundary metric and mean curvature.
The necessity of the boundary and static vacuum conditions
for a minimizer have been established
(see \cites{AJ},
 \cites{CorDef, CorBarNote, CorBarShort} and \cites{HMM},
 and \cites{MiaoVarEff},
 as well as \cites{HL, AnMin} 
 for the spacetime version),
but the existence of a minimizer is known
only
(at least to the author)
for
(i) apparent horizons,
by virtue of \cites{MS}
and its extension
\cites{ChauMartensDegenerate}
to the degenerate case,
and (ii)
for data which can be realized
as an outer-minimizing embedded sphere 
(required for our definition, \eqref{mbdef})
enclosing the horizon in a time-symmetric Schwarzschild slice
(or any outer-minimizing embedded sphere in Euclidean space),
by virtue of \cites{HI}
(or \cites{SY, Wit}).
Although the most aggressive formulations of the conjecture
are now known to be false
(see \cites{MiaoBdyEffStat, AJ, MS}
 and for counterexamples
 in the spacetime setting in higher dimensions
 also \cites{HL}),
it does suggest a strategy, pursued in this article,
for seeking a tighter upper bound than
the Brown-York mass affords in \eqref{sandwich}.
Our results are contained in the following theorem and corollary.

\begin{theorem}
  [Mass estimate;
   cf. \cites{MiaoStatExt, AndLoc, An, AnHuang}
   regarding existence and uniqueness]
\label{mainthm}
Let $M:=\{\vec{x} \in \R^3 \; : \; \abs{\vec{x}} \geq 1\}$,
let $\iota: \partial M \to M$ be the inclusion map
of $\partial M = \Sph^2$ in $M$,
and
let $\delta$ be the standard Euclidean metric on $M$,
so that $\iota^*\delta$ is the round metric of area $4\pi$
and $\partial M$ has mean curvature (by our conventions) $-2$
in $(M, \delta)$.
There exists $\epsilon_0>0$ such that
for any Riemannian metric
$\gamma = \iota^*\delta + \upp{\gamma}{1}$
and function $H = -2 + \upp{H}{1}$
on $\Sph^2$
with
$\epsilon
 :=\norm{\upp{\gamma}{1}}_{3,\alpha}+\norm{\upp{H}{1}}_{2,\alpha}
 < \epsilon_0
$
there is an asymptotically flat static vacuum metric $g$
on $M$,
unique up to diffeomorphism (for $\epsilon_0$ sufficiently small),
which
induces metric $\gamma$ and mean curvature $H$ on $\partial M$
and which
satisfies $\norm{g-\delta}_{3,\alpha,1} = O(\epsilon)$.
Moreover $g$ has mass
$\madm[g] = \upp{m}{1} + \upp{m}{2} + O(\epsilon^3)$,
where
\begin{equation}
\label{m1m2}
\begin{aligned}
  \upp{m}{1}
    &=
    \frac{1}{16\pi}\displaystyle{\int_{\partial M}} \left(
                        2\upp{H}{1}
                        -\upp{\gamma}{1}^\sigma_{\;\;\sigma}
                      \right)
      \qquad \mbox{and} \\
  \upp{m}{2}
    &=
    \frac{1}{16\pi}\int_{\partial M}
      \left[
        \upp{H}{1}\left(\upp{\gamma}{1}^\sigma_{\;\;\sigma}-f-v\right)
        +\frac{1}{2}(v-v_{,r})(v+2f)
        +\frac{1}{2}\abs{\trl{\upp{\gamma}{1}}}^2
      \right],
\end{aligned}
\end{equation}
with
  \begin{equation}
    \trl{\upp{\gamma}{1}}
    :=
    \upp{\gamma}{1}
      -\frac{1}{2}\upp{\gamma}{1}^{\sigma}_{\;\;\sigma}\iota^*\delta
  \end{equation}
the ($\iota^*\delta$) trace-free part of
$\upp{\gamma}{1}$,
$v: M \to \R$ the unique ($\delta$) harmonic function on $M$
vanishing at infinity and
satisfying
\begin{equation}
\label{vrrmainthm}
  \iota^*v_{,rr}
    =
    2\upp{H}{1}
      +\trl{\upp{\gamma}{1}}_{\alpha\beta}^{\;\;\;\;\;:\beta\alpha}
      -\frac{1}{2}(\Delta+2)\upp{\gamma}{1}^\sigma_{\;\;\sigma},
\end{equation}
and $f: \Sph^2 \to \R$ any solution
to
\begin{equation}
\label{fdef}
  (\Delta + 2)f = \upp{H}{1} - \iota^*(v-v_{,r}).
\end{equation}
The preceding equation has a solution for any data,
and the choice of solution does not affect the value of $\upp{m}{2}$.
In the above,
$v_{,r}$ and $v_{,rr}$
are the first and second derivatives of $v$ in the radial direction
(into $M$),
and the integrals, the raising of indices,
the norm operation,
the Laplacian $\Delta$,
and the covariant differentiation indicated by the colon
in front of indices
are all defined via
the round metric $\iota^*\delta$.
\end{theorem}

Most of Theorem \ref{mainthm} is old news:
existence was established first in \cites{MiaoStatExt}
assuming reflectional symmetry through each of the coordinate planes
and later in generality,
along with local uniqueness and in arbitrary dimension,
in \cites{AndLoc}, applying results from \cites{AndBdy,AndKhu}.
In fact a proof of existence and uniqueness
for the more general stationary vacuum extension problem
(with small data) has recently been achieved in \cites{An}
and, beyond the ball, for small perturbations of a broad
class of Euclidean domains in the static case
in \cites{AnHuang}.
The first-order mass estimate $\upp{m}{1}$
appears in \cites{Wbmoss}
and actually holds
for any sufficiently small vacuum extension
by virtue of identity \eqref{C} in Appendix \ref{massapp}.
The novel contribution of the present article
is the computation of the quadratic term $\upp{m}{2}$
in the expansion of the extension's mass,
which proceeds by an elaboration of the approach in \cites{Wbmoss}.

Clearly, for $\epsilon_0$ sufficiently small,
$\partial M$ is outer-minimizing
in the extension featured in Theorem \ref{mainthm}.
Consequently
(and independently of the validity
 of Bartnik's static vacuum extension conjecture
 in this setting)
the ADM mass of this extension is an upper bound for
the Bartnik mass
$\mb[\gamma, H]$
(as defined in \eqref{mbdef})
of its boundary.
(For an estimate of the Bartnik mass of
 data which is exactly (rather than approximately, as here)
 constant-mean-curvature but not necessarily almost round
 see \cites{CPCMM}.)
In particular we get the following
upper bound for the Bartnik mass of small metric spheres.

\begin{cor}
  [Upper bound on the quintic term of
   the Bartnik mass of small spheres]
\label{cor}
Let $p$ be a point in a Riemannian $3$-fold
with nonnegative scalar curvature
and for each $\tau>0$ let $B_\tau$ be the closed metric ball of radius $\tau$.
Then
(recalling \eqref{mbomdef} and \eqref{mbdef})
for $\tau$ sufficiently small
we have the estimate
\begin{equation}
\label{bupperbd}
  \mb[B_\tau]
  \leq
  \frac{\tau^3}{12}R
        +\frac{\tau^5}{2160}
          \left(30 \abs{\Ric}^2 - 25R^2 + 18\Delta R\right)
        +O(\tau^6),
\end{equation}
where $R$ is the scalar curvature
$\abs{\Ric}^2$ the squared norm of the Ricci curvature,
$\Delta$ the Laplacian,
and all terms are evaluated at $p$.
\end{cor}

\begin{remark}[Conjectural equality]
We have equality in Corollary \ref{cor}
in the event that Bartnik's static vacuum extension conjecture
holds in this restricted, weak-field regime.
\end{remark}

\begin{remark}
[Comparison with Hawking and Brown-York masses]
Writing $\ms[B_\tau]$ for the right-hand side of \eqref{bupperbd}
(representing the ADM mass of the static vacuum extension
 from Theorem \ref{mainthm}
 for $\partial B_\tau$)
and recalling the expansions \eqref{hexp} and \eqref{byexp},
we have
\begin{equation}
\begin{aligned}
  \ms[B_\tau]-\mh[B_\tau]
    &=
    \frac{\tau^5}{216}\left(3\abs{\Ric}^2-R^2\right) + O(\tau^6)
    \quad \mbox{and} \\
  \mby[B_\tau]-\ms[B_\tau]
  &=
  \frac{\tau^5}{4320}\left(12\abs{\Ric}^2+11R^2\right) + O(\tau^6),
\end{aligned}
\end{equation}
so that $\mh[B_\tau] \leq \ms[B_\tau] + O(\tau^6)$,
with equality only when $\Ric$ is spherically symmetric at $p$,
while $\ms[B_\tau] \leq \mby[B_\tau] + O(\tau^6)$,
with equality only when $B_\tau$ is flat at $p$.
\end{remark}

An alternative computation
of $\ms[B_\tau]$,
that avoids reference
to the details of the construction
of the extension,
has been given in \cites{HarvieWang}
(after the appearance of
the present article in preprint form).

\subsection*{Outline}
The mass estimate in Theorem \ref{mainthm}
is proved by carrying out the construction
of a static vacuum extension in \cites{Wbmoss}
and keeping track of the mass to second order
(in the size of the perturbation from the Bartnik data of the standard sphere).
In Section 2 we recall the static vacuum system
and review the analysis of the linearized problem
in \cites{Wbmoss}, with some refinements.
In Section 3 we achieve both the static vacuum conditions
and the boundary conditions to first order.
In Section 4 we achieve the static vacuum conditions to second order
(without sacrificing 
 the satisfaction of the boundary conditions to first order).
In Section 5 we complete the construction to second order,
enforcing also the boundary conditions.
In Section 6 we compute the mass of the approximate solution
constructed in the previous steps,
so estimating the mass of an exact static vacuum extension
to second order and completing the proof of Theorem \ref{mainthm}.
In Section 7 we at last apply the mass estimate of Theorem \ref{mainthm}
to prove Corollary \ref{cor}.

There are five appendices.
The first three present
standard identities and computations for ease of reference,
while the fourth consists of a straightforward but somewhat lengthy
calculation we could not find elsewhere,
and the fifth presents an alternative proof
and generalization of a step
performed in \cites{Wbmoss}.
Appendix A declares our conventions
for the Riemann curvature tensor
and includes some useful identities,
particularly in dimension $3$.
Appendix B
contains the response of Ricci curvature to conformal change of metric.
Appendix C is used to relate the variation of the boundary data
under change of ambient metric to the corresponding variation of the mass.
Appendix D computes the second variation
of the induced boundary metric and mean curvature
with respect to simultaneous ambient diffeomorphism and conformal change.
Finally, Appendix E concerns
the linearization of the prescribed Ricci equation
on Euclidean domains without boundary conditions
and related material on symmetric tensors
of prescribed support
and (compatibly supported) divergence .

\subsection*{Acknowledgments}
This project has received funding
from the European Research Council (ERC)
under the European
Union’s Horizon 2020 research and innovation programme
(grant agreement No. 947923).
I thank  Alessandro Carlotto for
making this funding possible
and for his general encouragement and support.
I am also grateful to
Po-Ning Chen,
Nicos Kapouleas,
Christos Mantoulidis,
and Pengzi Miao
for illuminating and motivating discussions
that helped me to complete this calculation.
Finally, I thank the referees
for their valuable feedback,
which has improved this article.

\section{The static vacuum system and its linearization}
With assumptions and notation as in Theorem \ref{mainthm}
we seek an asymptotically flat Riemannian metric $g$ on $M$
inducing on $\partial M$ metric $\gamma$
and mean curvature
(by our convention the divergence of the origin-directed unit normal) $H$
together with a real-valued function $\Phi$ on $M$
making $(M,g)$ static vacuum with static potential $\Phi$.
In other words $(g,\Phi)$ must satsify
(see for example \cites{CB, CorDef}) the static vacuum system
  \begin{equation}
  \label{svs}
      D^2_{ab}[g]\Phi - \Phi R_{ab}[g] = 0
      \qquad \mbox{and} \qquad
      \Delta[g]\Phi=0
  \end{equation}
($D^2_{ab}[g]$, $R_{ab}[g]$, and $\Delta[g]$
 denoting respectively
 the Hessian operator, Ricci curvature, and Laplacian
 associated to $g$)
subject to the boundary conditions
  \begin{equation}
      \iota^*g_{\alpha \beta}=\gamma_{\alpha \beta}
      \qquad \mbox{and} \qquad
      \Hcal[\iota,g]=H,
  \end{equation}
where $\Hcal[\iota,g]$ is the $g$ divergence of the outward
(meaning pointing away from $M$) $g$ unit normal for $\partial M$,
and we also impose
the decay conditions
  \begin{equation}
      g_{ab}-\delta_{ab} \in C^{3,\alpha,1}(T^*M^{\odot 2})
      \qquad \mbox{and} \qquad
      \Phi-1 \in C^{3,\alpha,1}(M).
  \end{equation}
We follow the procedure of \cites{Wbmoss},
with minor modifications and more detailed estimates.

We define the operators
  \begin{equation}
    \begin{aligned}
      &\Scal: \Met^{k,\alpha,\beta}(M)
          \times \left(1+C^{k,\alpha,\beta}(M)\right)
        \to C^{k-2,\alpha,\beta+2}\left(T^*M^{\odot 2}\right)
          \times C^{k-2,\alpha,\beta+2}(M)
        \quad \mbox{and} \\
      &\Bcal: \Met^{k,\alpha,\beta}(M)
        \to C^{k,\alpha}\left(T^*\partial M^{\odot 2}\right)
          \times C^{k-1,\alpha}(\partial M)
        \quad \mbox{by} \\
      &\Scal[g,\Phi]:=\left(D^2_{ab}[g]\Phi-\Phi R_{ab}[g],
			    \;\; \Delta[g]\Phi\right)
      \qquad \mbox{and} \qquad
      \Bcal[g]:=\left(\iota^*g,\Hcal[\iota,g]\right),
    \end{aligned}
  \end{equation}
and for a given metric $k$ and function $\Psi$ on $M$
we denote the linearization of $\Scal$ at $(k, \Psi)$
and the linearization of $\Bcal$ at $k$
by $\Sdot[k,\Psi]$ and $\Bdot[k]$ respectively.
In particular we have
\begin{equation}
\label{Sdot}
  \Sdot[\delta, 1](h, u)
  =
  \left(
    u_{;ab} - \dot{R}_{ab}[\delta](h),
      \;\;
    \Delta[\delta]u
  \right),
\end{equation}
where $u_{;ab}$ and $\Delta[\delta]u$
are the Hessian and Laplacian of $u$
with respect to the Euclidean metric $\delta$
and the linearization $\dot{R}_{ab}[\delta]$ at $\delta$
of the Ricci curvature
can be expressed in the well known form
\begin{equation}
\label{Rdot}
  \dot{R}_{ab}[\delta](h)
  =
  \frac{1}{2}
  \left(
    {{h_a}^c}_{;cb} + {{h_b}^c}_{;ca} - {h^c}_{c;ab} - {h_{ab;c}}^{;c}
  \right),
\end{equation}
the semicolons indicating differentiation via $\delta$.

To solve the boundary-value problem
$\Scal[g,\Phi]=(0,0)$
with $\Bcal[g]=(\gamma, H)$
(and the desired decay)
we will first analyze the linearized problem
and construct an exact solution to the nonlinear problem
by iteration
(or more formally by invoking the contraction mapping lemma,
 as in \cites{Wbmoss}).
We will split the linearized problem into two parts:
first the problem with homogeneous data for $\Sdot[\delta,1]$
and prescribed inhomogeneous boundary data for $\Bdot[\delta]$,
namely
\begin{equation}
\label{Shom}
  \Sdot[\delta, 1](h_{ab}, u) = (0, 0)
  \qquad \mbox{and} \qquad
  \Bdot[\delta](h_{ab}) = \left(\widetilde{\gamma}_{\alpha\beta}, 
				\widetilde{H}\right),
\end{equation}
and second the problem with no boundary constraint
and
for $\Sdot[\delta,1]$
prescribed inhomogeneous data $(S_{ab}, \sigma)$
modulo the Lie derivative of the Euclidean metric $\delta$
along some vector field $\chi$, another unknown of the system,
namely
\begin{equation}
\label{Sinhom}
  \Sdot[\delta, 1](h, u)
  =
  \left(S_{ab} + \chi_{a;b} + \chi_{b;a}, \; \sigma\right),
\end{equation}
the semicolon indicating differentation via $\delta$.

Since we already control the boundary data in \eqref{Shom},
we may ignore it when solving \eqref{Sinhom}.
The Lie derivative term in \eqref{Sinhom}
is necessary because of an obstruction
arising from the linearization of the twice contracted Bianchi identity;
it is acceptable because in every such problem we face in practice
this identity will be approximately satisfied by $S$ itself,
so that the error introduced by $\chi$ will be
higher-order
and can be safely absorbed by the next iteration of the scheme.
The details clarifying and justifying these assertions
will be reviewed below when needed.

We solve \eqref{Shom}
by deforming the Euclidean metric $\delta$ by
linearized diffeomorphism and conformal change
and accordingly altering the trivial potential $1$.
In doing so we encounter
the linear system featured in the following lemma
and relevant to the statement of Theorem \ref{mainthm}.

\begin{lemma}
[Analysis of the boundary system]
\label{keysyslemma}
Suppose $\alpha \in (0,1)$,
$\widetilde{\gamma} \in C^{3,\alpha}(\partial M)$,
$\widetilde{H} \in C^{2,\alpha}(\partial M)$.

\begin{enumerate}[(i)]
\item There exist functions
$v \in C^{3,\alpha,1}(M)$,
$f \in C^{4,\alpha}(\partial M)$,
and a vector field
$X^\sigma \in C^{4,\alpha}(\partial M)$
satisfying the system
\begin{equation}
\label{keysys}
\begin{aligned}
  &(a) \qquad \Delta_\delta v=0 \\
  &(b) \qquad \iota^*v+2f=\frac{1}{2}\widetilde{\gamma}^\sigma_{\;\; \sigma}
      - {X^\sigma}_{:\sigma} \\
  &(c) \qquad \iota^*(v-v_{,r})+(\Delta_{\iota*\delta}+2)f=\widetilde{H} \\
  &(d) \qquad \widetilde{\gamma}_{\mu \nu}
    =
    \left(\frac{1}{2}\widetilde{\gamma}^\sigma_{\;\; \sigma}
	  - {X^\sigma}_{:\sigma}\right)\iota^*\delta_{\mu \nu}
      +X_{\mu:\nu}+X_{\nu:\mu},
\end{aligned}
\end{equation}
where 
$\Delta_\delta=\Delta[\delta]$ is the flat Laplacian
on $M \subset \R^3$,
$\Delta_{\iota^*\delta}=\Delta[\iota^*\delta]$
is the round Laplacian on $\Sph^2=\partial M$,
and
the raising and lowering of indices
as well as the differentation indicated by the colon
are performed via the round metric $\iota^*\delta$.

\item The system
\eqref{keysys} and data $\widetilde{\gamma}$
and $\widetilde{H}$ uniquely determine $v$:
it is the unique harmonic function on $M$
vanishing at infinity and
satisfying
\begin{equation}
\label{vrr}
  \iota^*v_{,rr}
  =
  2\widetilde{H}
    + \trl{\widetilde{\gamma}}_{\rho\sigma}^{\;\;\;\;:\sigma\rho}
    - \frac{1}{2}(\Delta_{\iota^*\delta} + 2)\widetilde{\gamma}^{\sigma}_{\;\;\sigma},
\end{equation}
where
$\trl{\widetilde{\gamma}}
 :=
  \widetilde{\gamma}-\frac{1}{2}\widetilde{\gamma}^{\sigma}_{\;\;\sigma}
    \iota^*\delta
$
and $\iota^*v_{,rr}$
is the restriction to $\partial M = \Sph^2$
of the second radial derivative (directed into $M$) of $v$.
In particular
\begin{equation}
\label{vest}
  \norm{v}_{3,\alpha,1}
  \leq
  C\left(
    \norm{\widetilde{\gamma}}_{3,\alpha} + \norm{\widetilde{H}}_{2,\alpha}
  \right)
\end{equation}
for some $C>0$ independent of $\widetilde{\gamma}$ and $\widetilde{H}$.

\item Any two solutions to equation (d) in \eqref{keysys}
      differ by a conformal Killing field on $\Sph^2=\partial M$;
      any solution $X$ to (d), together with $v$
      satisfying (a) and \eqref{vrr},
      uniquely determines a function $f$
      so that \eqref{keysys} is satisfied.
\end{enumerate}
\end{lemma}

\begin{proof}
Item (i) was established in \cites{Wbmoss},
but we provide a more streamlined proof here
and fill in the remaining items.

Equation (d) in \eqref{keysys}
is thoroughly understood
in the literature
(\cites{York});
in this paragraph
we briefly summarize its standard analysis.
We can rewrite (d) as
$LX = \trl{\widetilde{\gamma}}$,
with $L$ the conformal Killing operator in two dimensions.
Taking the divergence of both sides
yields
$(L^*L X)_\alpha = \trl{\widetilde{\gamma}}_{\alpha\beta}^{\;\;\;\;\;:\beta}$,
with $L^*$ the formal adjoint of $L$.
Then $L^*L$ is second-order elliptic with the same kernel as $L$,
namely the space of conformal Killing fields,
to which the divergence of the trace-free symmetric tensor
$\trl{\widetilde{\gamma}}$ is orthogonal.
It follows that $L^*LX = \div \trl{\widetilde{\gamma}}$
has a solution $X^\sigma \in C^{4,\alpha}(\partial M)$,
so that $\trl{\widetilde{\gamma}}-LX$
has vanishing trace and divergence,
but on $2$-sphere
the only such symmetric $2$-tensor is the trivial one
(as evident from \eqref{dim2TT}, since the sphere
has strictly positive curvature).
Thus $X$ is a solution of equation (d) in \eqref{keysys}.
It is clear from equation (b)
that $f$ is then uniquely determined by $v$
and the choice of $X$,
completing the proof of (iii).

Solving (b) for $f$,
substituting the result into (c),
and rearranging
yields
$\Delta_{\iota^*\delta}\iota^*v + 2v_{,r}
 =
 -2\widetilde{H}
  +(\Delta_{\iota^*\delta}+2)(\frac{1}{2}\tr \gamma-\div X)
$,
but by
taking the double divergence of $LX=\trl{\widetilde{\gamma}}$
we get $(\Delta_{\iota^*\delta}+2) \div X=\div \div \trl{\widetilde{\gamma}}$.
By applying the expression
for the Laplacian in spherical coordinates
to (a)
we obtain \eqref{vrr}.
Conversely, if (a) and (d) hold,
and we define $f$ by (b),
then (c) holds too.
This completes the proof of (ii) and (i),
except for the existence of a harmonic $v \in C^{3,\alpha,1}(M)$
satisfying \eqref{vrr} and \eqref{vest}
(since the asserted regularity of $f$ then follows from equation (c)
 of \eqref{keysys}).

The Laplacian on $(\Sph^2,\iota^*\delta)$
has eigenvalues $-\ell(\ell+1)$ for $\ell \in \Z \cap [0,\infty)$,
and an eigenfunction $w$ with eigenvalue $-\ell(\ell+1)$
has the unique extension $wr^{-\ell-1}$ to a harmonic function
on $M$ vanishing at infinity.
Consequently,
since $(\ell+1)(\ell+2)>0$ for all $\ell \geq 0$,
we are guaranteed that
that there is a unique harmonic function $v$ on $M$
satisfying \eqref{vrr} and vanishing at infinity
and that its restriction to $\Sph^2$
lies at least in the Sobolev space $H^3(\partial M)$,
so certainly in $C^0(\partial M)$.
Moreover, the preceding paragraph shows that this restriction satisfies
\begin{equation}
  (\Delta_{\iota^*\delta} + 2\mathcal{N})\iota^*v
  =
  -2\widetilde{H}
  -\trl{\widetilde{\gamma}}_{\rho\sigma}^{\;\;\;\;\;\:\sigma\rho}
  +\frac{1}{2}(\Delta_{\iota^*\delta}+2)\widetilde{\gamma}^\sigma_{\;\;\sigma},
\end{equation}
where $\mathcal{N}$ is the Dirichlet-to-Neumann map
for the harmonic extension problem on $(M,\delta)$.
Since for any $k \geq 0$
$\mathcal{N}$ is a bounded linear map
from $C^{k+2,\alpha}(\partial M)$
to $C^{k+1,\alpha}(\partial M)$,
it follows by interpolation and elliptic regularity
that
$\norm{w}_{k+2,\alpha}
 \lesssim
 \norm{(\Delta+2\mathcal{N})w}_{k,\alpha}
 +\norm{w}_0
$
whenever $w \in C^{k+2,\alpha}(\partial M)$.
A standard argument using mollification
and compactness of $C^{k+2,\alpha}$ in $C^{k+2}$
completes the proof.
\end{proof}

In a moment Lemma \ref{keysyslemma}
will be applied to the linearized static vacuum boundary-value problem,
but it also ensures that the functions $f$ and $v$,
as well as the mass contribution $\upp{m}{2}$ itself,
appearing in the statement of Theorem \ref{mainthm}
are well defined.

\begin{cor}
[Well-definedness of the expression for $\upp{m}{2}$]
Make the assumptions of Theorem \ref{mainthm}.
(i) Equation \eqref{vrrmainthm} has a unique harmonic solution
in $M$ vanishing at infinity.
(ii) Equation \eqref{fdef} has a solution,
and the value of $\upp{m}{2}$ in \eqref{m1m2}
does not depend on the choice of solution.
\end{cor}

\begin{proof}
Take $\upp{\gamma}{1}$ and $\upp{H}{1}$
as in Theorem \ref{mainthm}
and apply Lemma \ref{keysyslemma}
with $\widetilde{\gamma}=\upp{\gamma}{1}$
and $\widetilde{H}=\upp{H}{1}$
and equations \eqref{vrrmainthm}
and \eqref{fdef}
supplemented
by equations (a), (b), and (d) in \eqref{keysys}.
Item (i) and the existence claim of item (ii) follow directly.
For the remaining claim in (ii)
suppose that
 $k$ lies in the kernel of $\Delta_{\iota^*\delta}+2$.
Then by equation (c) of \eqref{keysys}
\begin{equation}
  -\upp{H}{1}k + \frac{1}{2}(v-v_{,r})2k
  =
  -\upp{H}{1}k + \upp{H}{1}k - k(\Delta_{\iota*\delta}+2)f,
\end{equation}
whose integral over $\Sph^2=\partial M$ vanishes
in view of the assumption that
$k$ belongs to the kernel
of the symmetric operator $\Delta_{\iota^*\delta}+2$.
Referring to \eqref{m1m2},
this
confirms that the value of $\upp{m}{2}$
is independent of the choice of $f$ in \eqref{fdef}.
\end{proof}

Now we solve the linearized static vacuum boundary-value problem
\eqref{Shom}.

\begin{prop}[BVP with homogeneous interior data]
\label{hom}
Let $\alpha \in (0,1)$.
Given any boundary data
$\widetilde{\gamma} \in C^{3,\alpha}(\partial M)$
and $\widetilde{H} \in C^{2,\alpha}(\partial M)$
the linearized static vacuum boundary-value problem
\eqref{Shom}
has a solution $(h_{ab}, u)$
with
$\norm{h}_{3,\alpha,1} + \norm{u}_{3,\alpha,1}
 \leq
 C\left(\norm{\widetilde{\gamma}}_{3,\alpha}
	+\norm{\widetilde{H}}_{2,\alpha}
  \right) 
$
for some $C>0$ independent of the data
$\widetilde{\gamma}$ and $\widetilde{H}$.
The function $u$ is uniquely determined by the data,
and any two choices for $h_{ab}$
differ by the Lie derivative of the flat metric $\delta$
along a vector field
whose restriction to $\partial M$
(with values in $T\R^3$)
is the restriction to $\partial M$ of some Killing field
on $\R^3$.
\end{prop}

\begin{proof}
For existence we repeat the proof in \cites{Wbmoss}.
The static vacuum conditions \eqref{svs}
are obviously preserved by diffeomorphisms,
but
diffeomorphisms deforming the geometry of $\partial M$
will alter the boundary data.
At the linearized level
this means
(keeping in mind that the background potential $1$ is constant)
that for any vector field $\xi^a$ on $M$
the pair $({}_1h_{ab},{}_1u)=(\xi_{a;b}+\xi_{b;a},0)$
will satisfy $\Sdot[\delta,1]({}_1h,{}_1u)=(0,0)$.
Meanwhile $\Bdot$ will be sensitive only to $\xi|_{\partial M}$;
specifically,
if $\xi|_{\partial M}=\iota_*X+f\partial_r$
for some function $f$
and vector field $X^\alpha$ on $\Sph^2=\partial M$,
then
\begin{equation}
\label{Bdotvec}
  \Bdot[\delta]({}_1h)
  =
  (X_{\alpha:\beta}+X_{\beta:\alpha} + 2f\iota^*\delta,
    \;\;
   (\Delta+2)f).
\end{equation}

On the other hand,
it is also easy to use conformal changes
with harmonic conformal factor
to adjust the boundary data
while maintaining the static vacuum conditions.
Specifically, if $v$ is a harmonic function on $M$,
then examination of \eqref{Sdot} and \eqref{Rdot}
confirms that the pair
$({}_2h_{ab}, {}_2u)=(v\delta_{ab}, -v/2)$
satisfies $\Sdot[\delta,1]({}_2h,{}_2u)=(0,0)$,
and at the boundary we have
\begin{equation}
\label{Bdotconf}
  \Bdot[\delta]({}_2h)
  =
  (\iota^*(v\delta), \;\;  \iota^*(v-v_{,r})).
\end{equation}
Summing the contributions of
\eqref{Bdotvec} and \eqref{Bdotconf},
it is clear that Lemma \ref{keysyslemma}
delivers a solution to \eqref{Shom}
satisfying the asserted estimate.

For uniqueness suppose $(h_{ab}, u)$
is a solution to \eqref{Shom}
with $\widetilde{\gamma}_{\alpha\beta}=0$
and $\widetilde{H}=0$.
It follows from \eqref{Sdot} and \eqref{Rdot}
that $\dot{R}_{ab}[\delta](h-2u\delta)=0$.
By Proposition \ref{trivialityRicFlat}
$h_{ab}-2u\delta_{ab}=\xi_{a;b}+\xi_{b;a}$
for some $\xi^a$,
but by items (ii) and (iii) of Lemma \ref{keysyslemma}
$u=0$ and $\xi|_{\partial M}$ is the restriction
to $\partial M$ of some Killing field.
\end{proof}

To complete the analysis of the linearized problem
we now solve \eqref{Sinhom}.

\begin{prop}[Inhomogenous interior data]
\label{inhom}
Let $\alpha \in (0,1)$ and $\beta \in (1/2,1)$.
Given $S_{ab} \in C^{1,\alpha,3+\beta}(T^*M^{\odot 2})$
and $\sigma \in C^{1,\alpha,3+\beta}(M)$,
system \eqref{Sinhom} has a solution
$(h_{ab}, u, \chi^a)$
so that for some $C>0$ independent of the data we have
(i) $\norm{h}_{3,\alpha,\beta}
 +\norm{u}_{3,\alpha,1+\beta}
 \leq C(\norm{S}_{1,\alpha,3+\beta}
        +\norm{\sigma}_{1,\alpha,3+\beta})$,
(ii) $h_{ab}^{\;\;\;\; ;b}=\frac{1}{2}h^c_{\;\; c;a}$,
and
(iii) $\norm{\chi}_{2,\alpha, 2+\beta}
        \leq C\norm{\Delta_\delta \chi}_{0,\alpha,4+\beta}
      $.
\end{prop}

\begin{proof}
First pick $u$ satisfying
${u^{;c}}_{;c}=\sigma$
and the estimate required by item (i) of the statement.
To solve for $h$ (and $\chi$ in the process)
the proof of the corresponding Lemma 2.8 in \cites{Wbmoss}
made use of the representation \eqref{incompatibility}
and an associated result in elasticity theory
(see for example Section 17 of \cites{Gurtin}.)
Alternatively we can apply Proposition \ref{selfequilibration}
with $F_{ab}:=u_{;ab}-S_{ab}$
to select $\chi$, satisfying the estimate (iii),
and then conclude by applying
Proposition \ref{prescribingricci}
to obtain $h$ satisfying (i) and (ii).
\end{proof}

\section{Solving to first order}
With assumptions as in Theorem \ref{mainthm}
we start by taking $f, v, X$ as in Lemma \ref{keysyslemma}
with $\widetilde{\gamma}=\upp{\gamma}{1}$
and $\widetilde{H}=\upp{H}{1}$,
so that in particular
  \begin{equation}
  \label{vfest}
    \norm{v}_{3,\alpha,1}+\norm{f}_{4,\alpha} = O(\epsilon).
  \end{equation}
Next let $\xi^a$ be the unique
radially constant
($\delta$-parallel along rays extending from the origin)
vector field on $M$ satisfying
$\xi^a \circ \iota = {\iota^a}_{,\mu}X^\mu$.
(Here ${\iota^a}_{,\mu}$ is $d\iota$,
so that the right-hand side
is $X$ pushed forward by $\iota$,
while the left-hand side
is $\xi$ pulled back by $\iota$.)
Now set
\begin{equation}
  \Xi^a(\vec{x})
  :=
  \psi(\abs{\vec{x}})
      \left[f\left(\frac{\vec{x}}{\abs{\vec{x}}}\right)\vec{x} + \vec{\xi}(\vec{x})\right],
\end{equation}
where $\psi: [1,\infty) \to [0,1]$ is a smooth function
taking the value $1$ constantly on $[1,2]$
and taking the value $0$ constantly on $[3,\infty)$.
Thus, by Lemma \ref{keysyslemma}
(cf. the proof of Proposition \ref{hom}),
$(h_{ab}, u) := \Xi_{a;b}+\Xi_{b;a} + v\delta_{ab}, -v/2)$
is a solution to \ref{Shom}
with data $\left(\upp{\gamma}{1}, \upp{H}{1}\right)$.

We now define the
function $\rho: M \to \R$
and the map $\phi: M \to \R^3$ by
  \begin{equation}
  \label{rhophidef}
    \rho:=1+v
    \qquad \mbox{and} \qquad
    \phi(\vec{x})
    =
    \vec{x} + \Xi(\vec{x}).
  \end{equation}
For $\epsilon$ sufficiently small
$\rho$ is strictly positive
and $\phi$ is a diffeomorphism onto its image $\phi(M)$.
From these we define on $M$ the Riemannian metric and real-valued function
  \begin{equation}
  \label{g1Phi1def}
    \up{g}{1}_{ab}:=\rho \phi^*\delta_{ab}
    \qquad \mbox{and} \qquad
    \up{\Phi}{1}:=1-\frac{1}{2}v.
  \end{equation}
Clearly
  \begin{equation}
  \label{g1Phi1est}
    \norm{\up{g}{1}-\delta}_{3,\alpha,1} + \norm{\up{\Phi}{1}-1}_{3,\alpha,1}
    =
    O(\epsilon)
  \end{equation}
and by design
  \begin{equation}
  \label{quaderror}
    \norm{
      \Scal \left[ \up{g}{1},\up{\Phi}{1}\right]
    }_{1,\alpha,3+\beta}
      =O\left(\epsilon^2\right)
    \qquad \mbox{and} \qquad
    \norm{\Bcal \left[\up{g}{1}\right]}_{1,\alpha} 
    = (\gamma,H) + O\left(\epsilon^2\right),
  \end{equation}
  where we fix now and henceforth
  some $\beta \in (1/2,1)$.
  Here we use the facts
  (as we shall do repeatedly in the sequel)
  that,
  working in any coordinates,
  $\Bcal[g]$
  is pointwise
  a smooth function of the derivatives
  of $g$ from order zero to one
  and that,
  working in Cartesian coordinates,
  $\Scal[g,\Phi]$
  is pointwise
  a smooth function
  of the derivatives of $\Phi$
  from order zero to two
  and of the derivatives of $g$
  from order zero to two.
  Consequently we obtain
  \eqref{quaderror}
  by pointwise Taylor expansion
  to first order with a quadratic estimate
  for the remainder;
  for the decay estimate in particular
  we exploit the specific structure
  of $\Scal$.

\section{Achieving the static vacuum condition to second order}
We next seek higher-order corrections
$\upp{g}{2S} \in C^{3,\alpha,\beta}(T^*M^{\odot 2})$
and $\upp{\Phi}{2S} \in C^{3,\alpha,\beta}(M)$
making the pair
  \begin{equation}
    \left(\up{g}{2S},\up{\Phi}{2S}\right):=\left(\up{g}{1}+\upp{g}{2S},\up{\Phi}{1}+\upp{\Phi}{2S}\right)
  \end{equation}
static vacuum to second order in $\epsilon$ while preserving the boundary conditions to first order:
we require
  \begin{equation}
  \label{2S}
    \norm{\Scal\left[\up{g}{2S},\up{\Phi}{2S}\right]}_{1,\alpha,3+\beta}
    =O\left(\epsilon^3\right)
  \end{equation}
and
  \begin{equation}
  \label{etauest}
    \norm{\upp{g}{2S}}_{3,\alpha,\beta}+\norm{\upp{\Phi}{2S}}_{3,\alpha,\beta} = O\left(\epsilon^2\right).
  \end{equation}

Assuming we can arrange \eqref{etauest}, we will have
  \begin{equation}
    \Scal\left[\up{g}{2S},\up{\Phi}{2S}\right]
    =
    \Scal\left[\up{g}{1},\up{\Phi}{1}\right]
      +\dot{\Scal}[\delta,1]\left(\upp{g}{2S},\upp{\Phi}{2S}\right)
      +O\left(\epsilon^3\right)
  \end{equation}
(having also used
\eqref{g1Phi1est}
to justify writing
$\Sdot[\delta,1]$
rather than
$\Sdot[\up{g}{1}, \up{\Phi}{1}]$),
so to achieve \eqref{2S} we intend to solve
  \begin{equation}
    \dot{\Scal}[\delta,1]\left(\upp{g}{2S},\upp{\Phi}{2S}\right)
    =
    -\Scal\left[\up{g}{1},\up{\Phi}{1}\right] + O\left(\epsilon^3\right)
  \end{equation}
or equivalently
  \begin{equation}
  \label{Sto2nd}
    \begin{aligned}
      &D^2_{ab}[\delta]\upp{\Phi}{2S}
      -\dot{R}_{ab}[\delta]\upp{g}{2S} = \up{\Phi}{1}R_{ab}[g_1] - D^2_{ab}[g_1]\up{\Phi}{1} + O(\epsilon^3) \\
      &\Delta[\delta]\upp{\Phi}{2S}=-\Delta\left[\up{g}{1}\right]\up{\Phi}{1} + O(\epsilon^3)
    \end{aligned}
  \end{equation}
for $\upp{g}{2S}$ and $\upp{\Phi}{2S}$.

We obtain $\upp{g}{2S}$ and $\upp{\Phi}{2S}$
by applying Proposition \ref{inhom} with
  \begin{equation}
  \label{Ssigma}
    \begin{aligned}
      &S_{ab}:=\up{\Phi}{1} R_{ab}\left[\up{g}{1}\right] - D^2_{ab}\left[\up{g}{1}\right]\up{\Phi}{1},
      \qquad
      \sigma:=-\Delta\left[\up{g}{1}\right]\up{\Phi}{1}, \\
      &\upp{g}{2S}_{ab}:=h_{ab}, \qquad \mbox{and} \qquad
      \upp{\Phi}{2S}:=u.
    \end{aligned}
  \end{equation}
Note that since $\dot{R}_{ab}[\delta]\upp{g}{2S}$ itself satisfies
the linearization at $\delta$ of the twice contracted differential Bianchi identity,
we have
  \begin{equation}
    \chi_{a;b}^{\;\;\;\;\; ;b}
    =
    \frac{1}{2}\sigma_{;a}-S_{ab}^{\;\;\;\, ;b} + \frac{1}{2}S^b_{\;\; b;a},
  \end{equation}
but by \eqref{g1Phi1est} and \eqref{quaderror}
  \begin{equation}
    \norm{S_{ab}^{\;\;\;\, ;b}-{\up{g}{1}}^{bc}D_c\left[\up{g}{1}\right]S_{ab}}_{0,\alpha,4+\beta}
      +\norm{S^b_{\;\; b;a}-\left({\up{g}{1}}^{bc}S_{bc}\right)_{,a}}_{0,\alpha,4+\beta}
    = O\left(\epsilon^3\right),
  \end{equation}
so, applying the twice contracted differential Bianchi identity for $R_{ab}[g_1]$,
  \begin{equation}
    \begin{aligned}
      \chi_{a;b}^{\;\;\;\;\; ;b}
      &=
      -\frac{1}{2}\left(\Delta\left[\up{g}{1}\right]\up{\Phi}{1}\right)_{,a}
        -{\up{g}{1}}^{bc}{\up{\Phi}{1}}_{,c}R_{ab}\left[\up{g}{1}\right]
        -\frac{1}{2}\up{\Phi}{1}R_{,a}\left[\up{g}{1}\right]
        +\left(\Delta\left[\up{g}{1}\right]\up{\Phi}{1}\right)_{,a}
        +{\up{g}{1}}^{bc}{\up{\Phi}{1}}_{,c}R_{ab}\left[\up{g}{1}\right] \\
       &\;\;\;\;
        +\frac{1}{2}{\up{\Phi}{1}}_{,a}R\left[\up{g}{1}\right]
        +\frac{1}{2}\up{\Phi}{1} R_{,a}\left[\up{g}{1}\right]
        -\frac{1}{2}\left(\Delta\left[\up{g}{1}\right]\up{\Phi}{1}\right)_{,a}
        +O(\epsilon^3) \\
      &=
      \frac{1}{2}{\up{\Phi}{1}}_{,a}R\left[\up{g}{1}\right]+O(\epsilon^3).
    \end{aligned}
  \end{equation}

On the other hand, \eqref{g1Phi1est} and \eqref{quaderror}
imply that $\norm{R\left[\up{g}{1}\right]}_{1,\alpha,4} = O\left(\epsilon^2\right)$,
so in fact
  \begin{equation}
    \norm{\chi_{a;b}^{\;\;\;\;\; ;b}}_{0,\alpha,4+\beta} = O\left(\epsilon^3\right)
  \end{equation}
and therefore by item (iii) of Proposition \ref{inhom}
  \begin{equation}
    \norm{\chi}_{2,\alpha,2+\beta} = O\left(\epsilon^3\right),
  \end{equation}
so in turn
  \begin{equation}
    \norm{\chi_{a;b}+\chi_{b;a}}_{1,\alpha,3+\beta} = O\left(\epsilon^3\right).
  \end{equation}
We conclude that
  \begin{equation}
    \upp{\Phi}{2S}_{;ab}-\dot{R}_{ab}[\delta]\upp{g}{2S}
    =
    \up{\Phi}{1}R_{ab}\left[\up{g}{1}\right]-D^2_{ab}\left[\up{g}{1}\right]\up{\Phi}{1}+O\left(\epsilon^3\right),
  \end{equation}
which completes the verification of the system \eqref{Sto2nd},
the validity of whose second equation is obvious from the construction of
$\upp{\Phi}{2S}$ via Proposition \ref{inhom}.
We have now established \eqref{2S}.

\subsection*{An estimate for future use}
Before proceeding to correct the boundary metric and mean curvature to second order,
we pause to derive an estimate for $\int_M \upp{g}{2S}^{c \;\;\;\;\; ;d}_{\;\; c;d}$
that will be useful later.
It follows from item (ii) of Proposition \ref{inhom} that
$\dot{R}_{ab}[\delta]\upp{g}{2S}=-\frac{1}{2}\upp{g}{2S}_{ab;c}^{\;\;\;\;\;\; ;c}$,
whereby, in view of \eqref{Sto2nd},
  \begin{equation}
    \upp{g}{2S}_{ab;c}^{\;\;\;\;\;\; ;c}
    =
    2\up{\Phi}{1}R_{ab}\left[\up{g}{1}\right]-2D^2_{ab}\left[\up{g}{1}\right]\up{\Phi}{1}
      -2\upp{\Phi}{2S}_{;ab}+O\left(\epsilon^3\right).
  \end{equation}
Appealing again to \eqref{g1Phi1est} and \eqref{quaderror}
and contracting the above equation against $\delta^{ab}=\up{g}{1}^{ab}+O(\epsilon)$,
we find
  \begin{equation}
  \label{Deltatreta}
      \upp{g}{2S}^{c \;\;\;\;\; ;d}_{\;\; c;d}
      =
      2\up{\Phi}{1} R\left[\up{g}{1}\right]
        -2\Delta\left[\up{g}{1}\right]\up{\Phi}{1}-2\Delta[\delta]\upp{\Phi}{2S} + O(\epsilon^3)
      =
      2\up{\Phi}{1} R\left[\up{g}{1}\right] + O(\epsilon^3),
  \end{equation}
where to get the second estimate we have used \eqref{Sto2nd}.

Now
  \begin{equation}
    \begin{aligned}
      R_{ab}\left[\up{g}{1}\right]
      &=
      R_{ab}[\rho \phi^*\delta] \\
      &=
      \phi^*R_{ab}[\delta]
        +\frac{3}{4}\rho^{-2}\rho_{,a} \rho_{,b}
        -\frac{1}{2}\rho^{-1}D^2_{ab}[\phi^*\delta]\rho \\
       &\;\;\;\;
        -\frac{1}{2}\rho^{-1}\left(\Delta_{\phi^*\delta} \rho\right)(\phi^*\delta)_{ab}
        +\frac{1}{4}\rho^{-2}\abs{d\rho}_{\phi^*\delta}^2(\phi^*\delta)_{ab},
    \end{aligned}
  \end{equation}
so
  \begin{equation}
    \begin{aligned}
      R\left[\up{g}{1}\right]
      &=
      \frac{3}{4}\rho^{-3}\abs{d\rho}_{\phi^*\delta}^2
        -\frac{1}{2}\rho^{-2}\Delta_{\phi^*\delta}\rho
        -\frac{3}{2}\rho^{-2}\Delta_{\phi^*\delta}\rho
        +\frac{3}{4}\rho^{-3}\abs{d\rho}_{\phi^*\delta}^2 \\
      &=
      \frac{3}{2}\rho^{-3}\abs{d\rho}_{\phi^*\delta}^2
        -2\rho^{-2}\Delta_{\phi^*\delta}\rho \\
      &=
      \frac{3}{2}(1+v)^{-3}v_{,c}v_{,d}(\phi^*\delta)^{cd}
        -2(1-v)^2\Delta_{\phi^*\delta}v
        +O(\epsilon^3) \\
      &=
      \frac{3}{2}\abs{dv}_\delta^2
        -2\Delta_{\phi^*\delta}v
        +4v\Delta_{\phi^*\delta}v
        +O(\epsilon^3) \\
      &=
      \frac{3}{2}\abs{dv}_\delta^2
        -2\Delta_{\phi^*\delta}v
        +4v\Delta_{\delta}v
        +O(\epsilon^3) \\
      &=
      \frac{3}{2}\abs{dv}_\delta^2
        -2\Delta_{\phi^*\delta}v
        +O(\epsilon^3),
    \end{aligned}
  \end{equation}
having made use of \eqref{keysys}, \eqref{vfest}, and \eqref{rhophidef}.

Continuing \eqref{Deltatreta} with this last estimate, we get
  \begin{equation}
  \label{Deltatreta2}
    \begin{aligned}
      \upp{g}{2S}^{c \;\;\;\;\; ;d}_{\;\; c;d}
      &=
      2\left(1-\frac{1}{2}v\right)
        \left(\frac{3}{2}\abs{dv}_\delta^2-2\Delta_{\phi^*\delta}v\right)
        +O(\epsilon^3) \\
      &=
      3\abs{dv}_\delta^2 - 4\Delta_{\phi^*\delta}v + 2v\Delta_{\phi^*\delta}v + O(\epsilon^3) \\
      &=
      3\abs{dv}_\delta^2 - 4\Delta_{\phi^*\delta}v + 2v\Delta_{\delta}v + O(\epsilon^3) \\
      &=
      3\abs{dv}_\delta^2 - 4\Delta_{\phi^*\delta}v + O(\epsilon^3).
    \end{aligned}
  \end{equation}

Now
  \begin{equation}
  \label{firstint}
    \begin{aligned}
      \int_M \abs{dv}_\delta^2
      &=
      \int_M v^{;c}v_{;c}
      =
      \int_M \left(vv^{;c}\right)_{;c}
        -\int_M vv^{;c}_{\;\;\; ;c} \\
      &=
      -\int_{\partial M} vv_{,r}
         +\lim_{R \to \infty} \int_{\{r=R\}} vv_{,r}
         -0 \\
      &=
      -\int_{\partial M} vv_{,r},
    \end{aligned}
  \end{equation}
since $\norm{vv_r}_{0,0,3} = O(1)$
as a consequence of \eqref{vfest}.

Additionally,
writing $N_{\phi^*\delta}$ for the $\phi^*\delta$ $\infty$-directed unit normal to $\iota$,
  \begin{equation}
  \label{secondint_start}
    \begin{aligned}
      \int_M \Delta_{\phi^*\delta}v
      &=
      \int_{\{1 \leq r \leq 3\}} \Delta_{\phi^*\delta}v
      =
      \int_{\{1 \leq r \leq 3\}} \Delta_{\phi^*\delta}v \, \sqrt{\abs{\phi^*\delta}}
        +\int_{\{1 \leq r \leq 3\}} \Delta_{\phi^*\delta}v \, \left(\sqrt{\abs{\delta}}-\sqrt{\abs{\phi^*\delta}}\right) \\
      &=
      \int_{\{r=3\}} v_{,r} - \int_{\partial M} N_{\phi^*\delta}v \, \sqrt{\abs{\iota^*\phi^*\delta}}
        + \int_{\{1 \leq r \leq 3\}} \Delta_\delta v \, \left(\sqrt{\abs{\delta}}-\sqrt{\abs{\phi^*\delta}}\right)
        + O\left(\epsilon^3\right) \\
      &=
      \int_{\partial M} v_{,r} -  \int_{\partial M} N_{\phi^*\delta}v \, \sqrt{\abs{\iota^*\phi^*\delta}}
        +O\left(\epsilon^3\right) \\
      &=
      \int_{\partial M} \left(\partial_r - N_{\phi^*\delta}\right)v \, \sqrt{\abs{\iota^*\phi^*\delta}}
        +\int_{\partial M} v_{,r} \, \left(\sqrt{\abs{\iota^*\delta}}-\sqrt{\abs{\iota^*\phi^*\delta}}\right)
        +O\left(\epsilon^3\right) \\
      &=
      -\int_{\partial M} \left(N_{\phi^*\delta}-\partial_r \right)v
        +\int_{\partial M} v_{,r} \, \left(\sqrt{\abs{\iota^*\delta}}-\sqrt{\abs{\iota^*\phi^*\delta}}\right)
        +O\left(\epsilon^3\right),
    \end{aligned}
  \end{equation}
  where in particular we have used the facts that
  $N_{\phi^*\delta}-\partial_r = O(\epsilon)$,
  $v=O(\epsilon)$,
  and
  $
   \sqrt{\abs{\iota^*\phi^*\delta}}
   =
   \sqrt{\abs{\iota^*\delta}}
     +O(\epsilon)
  $
  to conclude that 
  the first integral on the penultimate line
  and the first integral on the ultimate line
  have $O(\epsilon^3)$ difference.
  More precisely, we compute 
  (for example by appealing to
  \eqref{normal}
  and the first line of
  \eqref{2ndvarsummary})
  \begin{equation}
  \begin{gathered}
  N_\phi^*\delta
    =
    \partial_r
    +d\iota(\nabla_{\iota^*\delta}f - X) 
    +O(\epsilon^2),
  \\
  \sqrt{\abs{\iota^*\phi^*\delta}}
    =
    \sqrt{\abs{\iota^*\delta}}
      +2f + \div_{\iota^*\delta}X
      +O(\epsilon^2).
  \end{gathered}
  \end{equation}
  Picking up where we left off
  in \eqref{secondint_start},
  we proceed to estimate
  \begin{equation}
  \label{secondint}
  \begin{aligned}
  \int_M \Delta_{\phi^*\delta}v
      &=
      \int_{\partial M} v_{:\sigma}\left(f^{:\sigma}-X^\sigma\right)
        -\int_{\partial M} v_{,r}\left(2f+{X^\sigma}_{:\sigma}\right)
        +O\left(\epsilon^3\right) \\
      &=
        -\int_{\partial M} v\left({f^{:\sigma}}_{:\sigma}-{X^{\sigma}}_{:\sigma}\right)
        -\int_{\partial M} v_{,r}\left(2f+{X^\sigma}_{:\sigma}\right)
          +O\left(\epsilon^3\right) \\
      &=
        -\int_{\partial M} v\left( {f^{:\sigma}}_{:\sigma} + v + 2f - \frac{1}{2}{\upp{\gamma}{1}^{\sigma}}_\sigma \right)
        -\int_{\partial M} v_{,r}\left(\frac{1}{2}{\upp{\gamma}{1}^{\sigma}}_\sigma - v \right)
        +O\left(\epsilon^3\right) \\
      &=
        -\int_{\partial M} v\left(\upp{H}{1}-\frac{1}{2}{\upp{\gamma}{1}^{\sigma}}_\sigma+v_{,r}\right)
        -\int_{\partial_M} v_{,r}\left(\frac{1}{2}{\upp{\gamma}{1}^{\sigma}}_\sigma - v \right)
        +O\left(\epsilon^3\right) \\
      &=
        -\int_{\partial M} \left[v\upp{H}{1}-\frac{1}{2}\left(v-v_{,r}\right){\upp{\gamma}{1}^{\sigma}}_\sigma \right]
          +O\left(\epsilon^3\right).
    \end{aligned}
  \end{equation}

From \eqref{Deltatreta2}, \eqref{firstint}, and \eqref{secondint}
we arrive at
  \begin{equation}
  \label{Deltatretaintest}
    \int_M \upp{g}{2S}^{c \;\;\;\;\; ;d}_{\;\; c;d}
    =
    \int_{\partial M} \left[4v\upp{H}{1}-2\left(v-v_{,r}\right){\upp{\gamma}{1}^{\sigma}}_\sigma-3vv_{,r}\right]
      +O\left(\epsilon^3\right).
  \end{equation}

\section{Solving to second order}
It remains to enforce the boundary data to second order
while maintaining the static vacuum condition also to second order:
we seek corrections
$\upp{g}{2B} \in C^{3,\alpha,\beta}\left(T^*M^{\odot 2}\right)$
and $\upp{\Phi}{2B} \in C^{3,\alpha,\beta}(M)$
such that the pair
  \begin{equation}
    \left(\up{g}{2},\up{\Phi}{2}\right)
    :=
    \left(\up{g}{2S}+\upp{g}{2B},\up{\Phi}{2S}+\upp{\Phi}{2B}\right)
  \end{equation}
satisfies
  \begin{equation}
    \Scal\left[\up{g}{2},\up{\Phi}{2}\right]=O\left(\epsilon^3\right)
    \qquad \mbox{and} \qquad
    \Bcal\left[\up{g}{2}\right]=(\gamma,H)+O\left(\epsilon^3\right).
  \end{equation}

Setting
  \begin{equation}
  \label{secondorderdefs}
    \begin{aligned}
      &\phi_t(\vec{x})
      :=
      \vec{x}
        +t\psi\left(\abs{\vec{x}}\right)
          \left[
            f\left(\frac{\vec{x}}{\abs{\vec{x}}}\right)+\vec{\xi}\left(\vec{x}\right)
          \right] \mbox{ (recalling \eqref{rhophidef}),} \\
      &\left(\upp{\gamma}{2B},\upp{H}{2B}\right)
      :=
      \frac{1}{2}\left.\frac{d^2}{dt^2}\right|_{t=0} \Bcal \left[(1+tv)\phi_t^*\delta \right],
      \quad \mbox{and} \quad
      \left(\upp{\gamma}{2S},\upp{H}{2S}\right)
      :=
      \dot{\Bcal}[\delta] \upp{g}{2S},
    \end{aligned}
  \end{equation}
by construction we have
  \begin{equation}
  \label{Bg2S_exp}
    \Bcal \left[\up{g}{2S}\right]
    =
    \Bcal \left[ \up{g}{1} \right]
      + \dot{\Bcal} \left[\up{g}{1}\right] \upp{g}{2S}
      + O\left(\epsilon^4\right)
    =
    \left(\gamma,H\right)
      + \left(\upp{\gamma}{2B},\upp{H}{2B}\right)
      + \left(\upp{\gamma}{2S},\upp{H}{2S}\right)
      +O\left(\epsilon^3\right)
  \end{equation}
(having used \eqref{etauest} in particular
for the first estimate)
and of course $\Scal \left[\up{g}{2S},\up{\Phi}{2S}\right]=O\left(\epsilon^3\right)$.
Provided
  \begin{equation}
  \label{2B_second-order}
    \norm{\upp{g}{2B}}_{3,\alpha,1+\beta} + \norm{\upp{\Phi}{2B}}_{3,\alpha,1+\beta} = O\left(\epsilon^2\right),
  \end{equation}
it follows that
  \begin{equation}
    \begin{aligned}
      \Scal \left[\up{g}{2},\up{\Phi}{2}\right]
      &=
      O\left(\epsilon^3\right)+\dot{\Scal}\left[\up{g}{2S},\up{\Phi}{2S}\right]\left(\upp{g}{2B},\upp{\Phi}{2B}\right)
      =
      \dot{\Scal}[\delta,1]\left(\upp{g}{2B},\upp{\Phi}{2B}\right) + O\left(\epsilon^3\right) \mbox{ and} \\
      \Bcal \left[\up{g}{2}\right]
      &=
      \Bcal\left[\up{g}{2S}\right]
        +\Bdot\left[\up{g}{2S}\right]
            \upp{g}{2B}
        +O(\epsilon^4)
      \\
      &=
      \left(\gamma,H\right)
      + \left(\upp{\gamma}{2B},\upp{H}{2B}\right)
      + \left(\upp{\gamma}{2S},\upp{H}{2S}\right)
      + \dot{\Bcal}[\delta]\upp{g}{2B}
      +O\left(\epsilon^3\right)
    \end{aligned}
  \end{equation}
(having used,
in addition to \eqref{2B_second-order},
the estimates
\eqref{g1Phi1est}
and
\eqref{etauest}
and, for the final line,
also the expansion
\eqref{Bg2S_exp}).

Accordingly we apply Proposition \ref{hom} to find 
$\left(\upp{g}{2B},\upp{\Phi}{2B}\right)$
solving
  \begin{equation}
  \label{2B}
    \Sdot[\delta,1]\left(\upp{g}{2B},\upp{\Phi}{2B}\right)=0
    \qquad \mbox{and} \qquad
    \Bdot[\delta]\left(\upp{g}{2B}\right)
    =
    -\left(\upp{\gamma}{2B},\upp{H}{2B}\right) - \left(\upp{\gamma}{2S},\upp{H}{2S}\right).
  \end{equation}
Note that \eqref{2B_second-order}
is then ensured,
in view of the definitions
\eqref{secondorderdefs}
and the estimates
\eqref{vfest}
and \eqref{etauest}.

\section{Mass estimate}
By iteratively correcting
the boundary and interior geometry
as in the above stages
(or applying the contraction mapping lemma
as in the proof of Proposition 2.51 in \cites{Wbmoss})
we obtain a static vacuum metric $(g,\Phi)$
with the prescribed boundary data,
  \begin{equation}
    \Scal[g,\Phi]=(0,0)
    \qquad \mbox{and} \qquad
    \Bcal[g]=(\gamma,H),
  \end{equation}
and satisfying
  \begin{equation}
    g=\up{g}{1}+\upp{g}{2S}+\upp{g}{2B}+O(\epsilon^3).
  \end{equation}
Consequently, 
  \begin{equation}
    \madm[g]
    =
    \madm\left[\up{g}{1}\right]
      +\madm\left[\delta+\upp{g}{2S}+\upp{g}{2B}\right]
      +O(\epsilon^3).
  \end{equation}

From \eqref{keysys}, \eqref{rhophidef}, and \eqref{g1Phi1def}
we get (as in \cites{Wbmoss})
  \begin{equation}
  \label{m1}
    \upp{m}{1}
    :=
    \madm\left[\up{g}{1}\right]
    =
    \frac{1}{16\pi}\int_{\partial M} -2v_{,r}
    =
    \frac{1}{16\pi}\int_{\partial M} \left(
                        2\upp{H}{1}
                        -\upp{\gamma}{1}^\sigma_{\;\;\sigma}
                      \right).
  \end{equation}
Similarly,
recalling \eqref{2B} and setting
  \begin{equation}
    \upp{m}{2B}
    :=
    \frac{1}{16\pi} \int_{\partial M}
      \left(\upp{\gamma}{2B}^\sigma_{\;\;\sigma}-2\upp{H}{2B}\right)
    \qquad \mbox{and} \qquad
    \upp{m}{2S}
    :=
    \frac{1}{16\pi} \int_{\partial M}
      \left(\upp{\gamma}{2S}^\sigma_{\;\;\sigma}-2\upp{H}{2S}\right),
  \end{equation}
we also have
  \begin{equation}
    \madm\left[\delta+\upp{g}{2B}\right]
    =
    \upp{m}{2B}+\upp{m}{2S}.
  \end{equation}

Using \eqref{t2g} and \eqref{t2H}
(and integration by parts) we obtain
  \begin{equation}
    \upp{m}{2B}
    =
    \frac{1}{16\pi} \int_{\partial M}
    \left[
      3(f+v)(\Delta+2)f
      +\frac{3}{2}v^2-3vv_{,r}
      +\abs{DX}^2-\abs{X}^2
      +2X^{\alpha:\beta}f_{:\alpha\beta}
      +2f {X^\sigma}_{:\sigma}
    \right].
  \end{equation}
By \eqref{keysys}(d) we have
  \begin{equation}
  \label{DeltaX}
    {X_{\alpha:\beta}}^{:\beta}+X_\alpha
    =
    \upp{\gamma}{1}_{\alpha\beta}^{\;\;\;\;\::\beta}
      -\frac{1}{2}\upp{\gamma}{1}^\sigma_{\;\;\sigma:\alpha}
    =
    \trl{\upp{\gamma}{1}}_{\alpha\beta}^{\;\;\;\;\;:\beta},
  \end{equation}
where
  \begin{equation}
    \trl{\upp{\gamma}{1}}
    :=
    \upp{\gamma}{1}
      -\frac{1}{2}\upp{\gamma}{1}^{\sigma}_{\;\;\sigma}\iota^*\delta.
  \end{equation}
Thus further integration by parts
and another appeal to \eqref{keysys}(d)
yields
  \begin{equation}
    \int_{\partial M}
      \left[\abs{DX}^2-\abs{X}^2\right]
    =
    -\int_{\partial M}
      \left[{X_{\alpha:\beta}}^{:\beta}+X_\alpha\right]X^\alpha
    =
    \int_{\partial M} \trl{\upp{\gamma}{1}}^{\alpha\beta}X_{\alpha:\beta}
    =
    \frac{1}{2}\int_{\partial M} \abs{\trl{\upp{\gamma}{1}}}^2.
  \end{equation}
Additionally, noting that $\partial M$
is the round unit sphere with Ricci curvature $\iota^*\delta$,
  \begin{equation}
    \int_{\partial M} X^{\alpha:\beta}f_{:\alpha\beta}
    =
    -\int_{\partial M} {X^\alpha}_{:\beta\alpha}f^{:\beta}
    =
    -\int_{\partial M}
      \left[
        {X^{\alpha}}_{:\alpha\beta}f^{:\beta}
        +X_\beta f^{:\beta}
      \right]
    =
    \int_{\partial M}
      {X^\sigma}_{:\sigma}(\Delta + 1)f.
  \end{equation}
It follows that
  \begin{equation}
    \upp{m}{2B}
    =
    \frac{1}{16\pi} \int_{\partial M}
    \left[
      (3f+3v+2{X^\sigma}_{:\sigma})(\Delta+2)f
      +\frac{3}{2}v^2-3vv_{,r}
      +\frac{1}{2}\abs{\trl{\upp{\gamma}{1}}}^2
    \right].
  \end{equation}
Using \eqref{keysys} we conclude
  \begin{equation}
    \upp{m}{2B}
    =
    \frac{1}{16\pi} \int_{\partial M}
    \left[
      \left(v-f+\upp{\gamma}{1}^\sigma_{\;\;\sigma}\right)
        \left(\upp{H}{1}-(v-v_{,r})\right)
      +\frac{3}{2}v^2-3vv_{,r}
      +\frac{1}{2}\abs{\trl{\upp{\gamma}{1}}}^2
    \right].
  \end{equation}

On the other hand,
by the definition of
$\left(\upp{\gamma}{2S},\upp{H}{2S}\right)$
in terms of $\upp{g}{2S}$ in \eqref{secondorderdefs},
the definition of $\upp{g}{2S}$ in \eqref{Ssigma}
via Proposition \ref{inhom}
(particularly item (ii)),
and the result of \eqref{C} (in Appendix \ref{massapp})
  \begin{equation}
  \begin{aligned}
    16\pi \upp{m}{2S} + O(\epsilon^4)
    &=
    -\int_{\partial M} \left(
        \upp{g}{2S}_{ij}^{\;\;\; ;j}-{\upp{g}{2S}^j}_{j;i}
      \right)x^i
      =
      \frac{1}{2} \int_{\partial M} {\upp{g}{2S}^j}_{j;i}x^i \\
      &=
      -16\pi\madm[\delta+\upp{g}{2S}]
      -\frac{1}{2} \int_M \upp{g}{2S}^{c \;\;\;\;\; ;d}_{\;\; c;d}.
  \end{aligned}
  \end{equation}
In turn \eqref{Deltatretaintest} yields
  \begin{equation}
    \upp{m}{2S} + \madm[\delta+\upp{g}{2S}]
    =
    \frac{1}{16\pi}\int_{\partial_M}
      \left[
        \frac{3}{2}vv_{,r}+(v-v_{,r})\upp{\gamma}{1}^\sigma_{\;\;\sigma}
          -2v\upp{H}{1}
      \right] + O(\epsilon^3).
  \end{equation}
Summing, we obtain the estimate
$\upp{m}{2}$ of Theorem \ref{mainthm}.

\section{Application to small spheres}
We will now apply our estimate to the case that
the data $(\gamma,H)$ correspond
to the boundary of a small metric ball,
appropriately scaled.
Specifically,
fix a Riemannian $3$-manifold $(N,h)$
and a point $p \in N$;
identify $\Sph^2=\partial M$
with the unit sphere in $T_pN$
and
for each $\tau \in \R$
define $\varphi_\tau: \Sph^2 \to N$
by $\varphi_\tau(v):=\exp^{(N,h)}_p \tau v$
($\exp_p^{(N,h)}: T_pN \to N$
being the exponential map of $(N,h)$ at $p$);
and (for $\tau>0$ sufficiently small) set
  \begin{equation}
    (\gamma,H)=(\gamma[\tau], H[\tau])
    :=
    (\varphi_\tau^*\tau^{-2}h, \Hcal[\varphi_\tau, \tau^{-2}h]),
  \end{equation}
so that the metric ball in $(N,h)$
of center $p$ and radius $\tau$
has induced metric $\tau^2 \gamma[\tau]$
and mean curvature $\tau H[\tau]$,
both pulled back to $\partial M$.

We then have (see for example \cites{Wbmoss})
the well known Taylor expansions
  \begin{equation}
  \label{Taylor}
    \begin{aligned}
      &\upp{\gamma}{1}_{\alpha\beta}
      =
      \frac{1}{3}\tau^2R_{\alpha r \beta r}
        +\frac{1}{6}\tau^3R_{\alpha r \beta r|r} 
        +\tau^4
          \left(
            \frac{1}{20}R_{\alpha r \beta r | rr}
            +\frac{2}{45}R_{\alpha rcr}{R^c}_{r \beta r}
          \right) + O(\tau^5) \mbox{ and} \\
      &\upp{H}{1}
      =
      \frac{1}{3}\tau^2R_{rr}
        +\frac{1}{4}\tau^3R_{rr|r}
        +\tau^4
          \left(
            \frac{1}{10}R_{rr|rr}+\frac{1}{45}R_{crdr}R^{crdr}
          \right) + O(\tau^5),
    \end{aligned}
  \end{equation}
where $R_{abcd}$ is the Riemann curvature tensor
(and $R_{ab}$ the Ricci curvature)
of $(N,h)$
(following the curvature conventions
declared in \eqref{curvatureconvention})
evaluated at $p$,
the vertical bar $|$
indicates differentiation via
the Levi-Civita connection induced by $h$,
also evaluated at $p$,
and each $r$ index indicates contraction with
$\nu:=\partial_r \circ \iota$,
the unit normal of $\partial M$
directed into $M$ and regarded as an element of $T_pN$.
Alternatively,
we may regard the curvature tensors
(and their contractions and derivatives)
in the above expansions
as parallel tensors over $\Sph^2=\partial M$,
writing for example
  \begin{equation}
    \upp{\gamma}{1}_{\alpha\beta}
    =
    \frac{1}{3}\tau^2 R_{abcd}{\iota^a}_{,\alpha}
      \nu^b{\iota^c}_{,\beta}\nu^d
      + \cdots,
  \end{equation}
etc.
We may also choose to specify a point on $\partial M$
by its standard Euclidean coordinates,
 or position vector, in $\R^3$
and thereby write
  \begin{equation}
    \upp{\gamma}{1}_{\alpha\beta}(\vec{x})
    =
    \frac{1}{3}\tau^2R_{\alpha i \beta j}x^ix^j + \cdots,
  \end{equation}
etc.
In the following we will take the liberty of applying
whichever of these notational options
best suits our purposes at a given step.

We wish to estimate
the ADM mass $m=m[\tau]=m[\gamma,H]$
of the unique
small static vacuum extension 
of $(\gamma, H)$,
within an error of order $\tau^5$.
Since $\norm{\upp{\gamma}{1}, \upp{H}{1}}_{\mathcal{B}}=O(\tau^2)$,
we have
  \begin{equation}
    m=\upp{m}{1}+\upp{m}{2}+O(\tau^6),
  \end{equation}
recalling \eqref{m1m2}
From \eqref{Taylor}
it follows readily (as in \cites{Wbmoss}) that
  \begin{equation}
    \upp{m}{1}
    =
    \upp{m}{1}[\gamma,H]
    =\frac{1}{12}R\tau^2 + \frac{1}{120}\Delta R \tau^4
      + O(\tau^5)
  \end{equation}
(where $R$ is the scalar curvature of $(N,h)$ at $p$).
To achieve our goal
we will now estimate $\upp{m}{2}$
up to order $\tau^4$,
which means we need only keep track
of the leading terms in \eqref{Taylor}.

From \eqref{keysys},
making use of \eqref{DeltaX}
and the identity
$\iota^*\Delta_\delta=\Delta_{\iota^*\delta}\iota^*
 +2\iota^*\partial_r+\iota^*\partial_r^2$,
we find
  \begin{equation}
    \iota^*v_{,rr}
    =
    2\upp{H}{1}
      +\trl{\upp{\gamma}{1}}_{\alpha\beta}^{\;\;\;\;\;:\beta\alpha}
      -\frac{1}{2}(\Delta+2)\upp{\gamma}{1}^\sigma_{\;\;\sigma}.
  \end{equation}
By \eqref{Taylor}
  \begin{equation}
    \begin{aligned}
      &2\upp{H}{1}+O(\tau^3)=\frac{2}{3}\tau^2R_{rr}, \\
      &\upp{\gamma}{1}^\sigma_{\;\;\sigma} + O(\tau^3)
      =
      -\frac{1}{3}\tau^2R_{rr}, \mbox{ and} \\
      &\upp{\gamma}{1}_{\alpha\beta}^{\;\;\;\;\;:\beta}
        +O(\tau^3)
      =
      \frac{1}{3}\tau^2R_{abcd}
       ({\iota^a}_{,\alpha}\nu^b
          {\iota^c}_{,\beta}\nu^d)^{:\beta}
      =
      \frac{1}{3}\tau^2R_{\alpha r},
    \end{aligned}
  \end{equation}
and so
  \begin{equation}
  \begin{aligned}
    &\upp{\gamma}{1}_{\alpha\beta}^{\;\;\;\;\;:\beta\alpha}
      +O(\tau^3)
    =
    \frac{1}{3}\tau^2(R_{ab}{\iota^a}_{,\alpha}\nu^b)^{:\alpha}
    =
    \frac{1}{3}\tau^2(-2R_{rr}+{R_\sigma}^\sigma)
    =
    \left(\frac{1}{3}R-R_{rr}\right)\tau^2, \\
    &-\frac{1}{2}\Delta \upp{\gamma}{1}^\sigma_{\;\;\sigma}
      +O(\tau^3)
    =
    \frac{1}{6}(R_{ab}\nu^a\nu^b)^{:\sigma}_{\;\;:\sigma}
    =
    \left(\frac{1}{3}R-R_{rr}\right)\tau^2, \mbox{ and} \\
    &\tau^{-2}\iota^*v_{,rr} + O(\tau)
    =
    R-2R_{rr}
    =
    \left(R-2R_{ij}x^ix^j\right)
    =
    \frac{1}{3}R+\left(\frac{2}{3}R-2R_{ij}x^ix^j\right),
  \end{aligned}
  \end{equation}
where in the last step
we have decomposed the leading terms of $v_{,rr}$
into spherical harmonics.
Recalling that $v$ is a harmonic function on $M$ vanishing at infinity,
we conclude that
  \begin{equation}
  \begin{aligned}
    &\iota^*v+O(\tau^3)
    =
    \frac{1}{6}R\tau^2
      +\frac{1}{12}\left(\frac{2}{3}R-2R_{ij}x^ix^j\right)\tau^2
    =
    \left(\frac{2}{9}R-\frac{1}{6}R_{rr}\right)\tau^2 \\
    &\iota^*v_{,r}+O(\tau^3)
    =
    -\frac{1}{6}R\tau^2
      -\frac{1}{4}\left(\frac{2}{3}R-2R_{ij}x^ix^j\right)\tau^2
    =
    \left(-\frac{1}{3}R+\frac{1}{2}R_{rr}\right)\tau^2, \mbox{ and} \\
    &\iota^*(v-v_{,r}) + O(\tau^3)
    =
    \left(\frac{5}{9}R-\frac{2}{3}R_{rr}\right)\tau^2.
  \end{aligned}
  \end{equation}

Next we will solve for $f$.
Referring again to \eqref{keysys},
we see that
  \begin{equation}
    (\Delta+2)f=\upp{H}{1}-\iota^*(v-v_{,r})
  \end{equation}
Thus, from the above,
  \begin{equation}
    (\Delta+2)f + O(\tau^3)
    =
    \left(-\frac{5}{9}R+R_{rr}\right)\tau^2
    =
    -\frac{2}{9}R\tau^2 + \left(-\frac{1}{3}R+R_{ij}x^ix^j\right)\tau^2,
  \end{equation}
and so
  \begin{equation}
    f+O(\tau^3)
    =
    -\frac{1}{9}R\tau^2-\frac{1}{4}\left(-\frac{1}{3}R+R_{rr}\right)\tau^2
    =
    \left(-\frac{1}{36}R-\frac{1}{4}R_{rr}\right)\tau^2.
  \end{equation}

With the preceding in place we can now compute
  \begin{equation}
  \begin{aligned}
    &\upp{\gamma}{1}^\sigma_{\;\;\sigma}-f-\iota^*v
    =
    \left(-\frac{7}{36}R+\frac{1}{12}R_{rr}\right)\tau^2
      +O(\tau^3), \\
    &\iota^*v+2f
    =
    \left(\frac{1}{6}R-\frac{2}{3}R_{rr}\right)\tau^2+O(\tau^3), \\
    &\upp{H}{1}
      \left(\upp{\gamma}{1}^\sigma_{\;\;\sigma}-f-\iota^*v\right)
    =
    \left(-\frac{7}{108}RR_{rr}+\frac{1}{36}R_{rr}^2\right)\tau^4
      +O(\tau^5), \\
    &\frac{1}{2}(v-v_{,r})(v+2f)
    =
    \left(
      \frac{5}{108}R^2-\frac{13}{54}RR_{rr}+\frac{2}{9}R_{rr}^2
    \right)\tau^4+O(\tau^5), \mbox{ and} \\
    &\frac{1}{2}\abs{\trl{\upp{\gamma}{1}}}^2
    =
    \frac{1}{2}\abs{\upp{\gamma}{1}}^2
      -\frac{1}{4}\left(\upp{\gamma}{1}^\sigma_{\;\;\sigma}\right)^2
    =
    \left(
      \frac{1}{18}R_{arbr}R^{arbr}-\frac{1}{36}R_{rr}^2
    \right)\tau^4+O(\tau^5).
  \end{aligned}
  \end{equation}

In turn we find
  \begin{equation}
    \upp{m}{2}
    =
    \frac{\tau^4}{16\pi} \int_{\Sph^2}
      \left[
        \left(
          \frac{1}{18}R_{aibj}R^{akb\ell}+\frac{2}{9}R_{ij}R_{k\ell}
        \right)x^ix^jx^kx^\ell
          +\frac{5}{108}R^2 - \frac{11}{36}RR_{ij}x^ix^j
      \right] + O(\tau^5).
  \end{equation}
Using
  \begin{equation}
    \begin{aligned}
      &\frac{1}{16\pi} \int_{\Sph^2} 1 = \frac{1}{4},
      \qquad
      \frac{1}{16\pi} \int_{\Sph^2} x^ix^j
        =\frac{1}{12}\delta^{ij},
      \quad \mbox{and} \\
      &\frac{1}{16\pi} \int_{\Sph^2} x^ix^jx^kx^\ell
        =
        \frac{1}{60}
          \left(  
            \delta^{ij}\delta^{k\ell}
            +\delta^{ik}\delta^{j\ell}
            +\delta^{i\ell}\delta^{jk}
          \right),
    \end{aligned}
  \end{equation}
we obtain
  \begin{equation}
    \upp{m}{2}
    =
    \frac{1}{1080}\abs{\operatorname{Riem}}^2\tau^4
        +\frac{1}{1080}R_{abcd}R^{adcb}\tau^4
        +\frac{1}{120}\abs{\Ric}^2\tau^4
        -\frac{11}{1080}R^2\tau^4
        +O(\tau^5).
  \end{equation}
Applying \eqref{Riemquad} we conclude
  \begin{equation}
    \upp{m}{2}
    =
    \frac{1}{72}\abs{\Ric}^2\tau^4-\frac{5}{432}R^2\tau^4+O(\tau^5)
    =
    \frac{1}{432}\left(6\abs{\Ric}^2-5R^2\right)\tau^4+O(\tau^5),
  \end{equation}
completing the proof of Corollary \ref{cor}.

\appendix

\section{Riemann curvature conventions and identities}
\label{Riem}
We adopt the convention that the Riemann curvature tensor
$R_{abcd}=\operatorname{Riem}$ of a Riemannian manifold $(M,g)$ satisfies
  \begin{equation}
  \label{curvatureconvention}
    {X^d}_{|ba}-{X^d}_{|ab}
    =
    X^c{R_{abc}}^d
  \end{equation}
for all smooth vector fields $X$
(the vertical bar indicating differentation
via the Levi-Civita connection
with respect to the indices following it).
Then $R_{abcd}=-R_{bacd}=-R_{abdc}=R_{cdab}$.
We denote the corresponding Ricci curvature
by $\Ric=R_{ab}={R_{cab}}^c=R_{ac\;\;\;b}^{\;\;\;\;c}$
and the scalar curvature by $R$.
For any given twice-differentiable symmetric
tensor $h_{ab}$ one readily computes the
linearized Riemann curvature 
\begin{equation}
\label{linRiem}
\begin{aligned}
  \dot{R}_{abc}^{\;\;\;\;\;d}[g]h
  &:=
  \left.\frac{d}{dt}\right|_{t=0}
    R_{abc}^{\;\;\;\;\;d}[g+th] \\
  &=
  \frac{1}{2}
  \left(
    {h^d}_{b|ca}
    +h_{ac\;\;\;|b}^{\;\;\;\;|d}
    -{h^d}_{a|cb}
    -h_{bc\;\;\;|a}^{\;\;\;\;|d}
    +{R_{abf}}^d {h_c}^f
    +R_{abfc} h^{df}
  \right).
\end{aligned}
\end{equation}

\subsection*{Dimension 3}
Now assume that $M$ has dimension $3$
and let $\epsilon_{abc}$
be a choice of orientation form (at least locally defined).
A short calculation reveals the identity
  \begin{equation}
  \label{epsilonG}
    \epsilon_{abx}\epsilon_{cdy}R^{abcd}=4G_{xy},
  \end{equation}
where $G_{ab}:=R_{ab}-\frac{1}{2}Rg_{ab}$ is the Einstein tensor,
and in turn we have
  \begin{equation}
    R_{abcd}=\epsilon_{abx}\epsilon_{cdy}G^{xy}
    =
    R_{ad}g_{bc}+R_{bc}g_{ad}-R_{ac}g_{bd}-R_{bd}g_{ac}
    +\frac{1}{2}Rg_{ac}g_{bd}-\frac{1}{2}Rg_{ad}g_{bc}.
  \end{equation}
In particular
  \begin{equation}
  \label{Riemquad}
    \abs{\operatorname{Riem}}^2=4\abs{\Ric}^2-R^2
    \qquad \mbox{and} \qquad
    R_{abcd}R^{adcb}=2\abs{\Ric}^2-\frac{1}{2}R^2.
  \end{equation}
Furthermore, it follows from \eqref{linRiem}
and \eqref{epsilonG}
that for $M \subset \R^3$
the linearization $\dot{G}[\delta]$
of the Einstein tensor about the Euclidean
metric $\delta$ is given by
\begin{equation}
\label{incompatibility}
  \dot{G}_{xy}[\delta]h
  =
  \frac{1}{2}\epsilon_{abx}\epsilon_{cdy}h^{bd|ca}.
\end{equation}

\begin{prop}
[Triviality of first-order Ricci-flat deformations]
\label{trivialityRicFlat}
Let $\Omega$ be a simply connected
open subset of $\R^3$.
If $h_{ab}$ is a symmetric tensor on $\Omega$
such that $\dot{R}_{ab}[\delta]h=0$,
then there is a vector field $X^a$
on $\Omega$ such that
$h_{ab} = X_{a;b} + X_{b;a} = L_X \delta_{ab}$.
\end{prop}

\begin{proof}
The claim can be established
using the representation
\eqref{incompatibility}
and the Poincar\'{e} lemma;
see for example Section 14 of \cites{Gurtin}.
Alternatively,
for any given $h_{ab} \in C^2_{loc}(T^{\odot 2}\Omega)$
and $p \in \Omega$
there is a sufficiently small
origin-centered open ball $B$
in $\R^3$ such that for all sufficiently small $t$
each exponential map $\phi_t$ of $g_t := \delta + th$
at $p$ is a diffeomorphism of $B$ onto its image.
Assuming $\dot{R}_{ab}[\delta]h=0$,
we have also (in dimension $3$)
$\dot{R}_{abcd}[\delta]h=0$,
but then
$\left.\frac{d}{dt}\right|_{t=0}\phi_t^*g_t = 0$,
meaning that $h=L_X \delta$
on $B$
with $X=-\left.\frac{d}{dt}\right|_{t=0}\phi_t$.
Furthermore,
if $Y$ and $Z$ are vector fields
on a connected open subset $\mathcal{U} \subset \Omega$
with $h=L_Y\delta$ on $\mathcal{U}$,
then $h=L_Z\delta$ on $\mathcal{U}$ as well if 
and only if $Y-Z$ is the restriction
of a Killing field on $\R^3$.
The existence of a global $X$ on $\Omega$
such that $h=L_X\delta$
now follows from the assumption
that $\Omega$ is simply connected.
\end{proof}

\subsection*{Dimension 2}
Note that,
since the Einstein tensor vanishes in dimension $2$,
it follows directly from \eqref{linRiem}
that any twice differentiable
transverse (that is having vanishing divergence)
traceless symmetric tensor $\eta_{ab}$ on $(M,g)$
satisfies
\begin{equation}
\label{dim2TT}
  {\eta_{ab|c}}^{|c} - 2R\eta_{ab} = 0.
\end{equation}

\section{Ricci curvature under conformal change of metric}
We will derive the well-known expression
for the transformation of Ricci curvature
under conformal change of metric
by applying the standard formula
for the first variation of mean curvature
under normal deformation.
Let $(M,g)$ be a Riemannian manifold
of dimension $\dim M = n+1$,
and let $\rho \in C^2_{loc}(M)$ be strictly positive.
We will make use of the identities
\begin{align}
\label{Laplacianconf}
  &\Delta_{\rho g} u
  =
  \rho^{-1}\Delta_g u
    + \frac{\dim M - 2}{2}\rho^{-2}g^{cd}\rho_{|c}u_{|d}, \\
\label{restrictedLaplacian}
  &\left(\Delta_g u\right)|_\Sigma
  =
  \Delta_{\iota^*g}u|_\Sigma
    + u_{|ab}|_\Sigma \nu^a \nu^b - Hu_{|a}|_\Sigma \nu^a, \\
\label{AHconf}
  &A[\iota,\rho g]_{\alpha \beta}
  =
  (\iota^*\rho)^{1/2}A[\iota,g]_{\alpha \beta}
   -\frac{1}{2}(\iota^*\rho)^{-1/2}\nu^c
     (\rho_{|c} \circ \iota)(\iota^*g)_{\alpha \beta},
     \quad \mbox{and} \\
 &H[\iota,\rho g]
  =
 (\iota^*\rho)^{-1/2}H[\iota,g]
   -\frac{n}{2}(\iota^*\rho)^{-3/2}
     \nu^c(\rho_{|c} \circ \iota).
\end{align}
where $u \in C^2_{loc}(M)$,
vertical bars before indices 
indicate covariant differentiation defined by $g$,
and $\Sigma$ is an embedded hypersurface of $M$,
with inclusion map $\iota$,
$\nu$ a local choice of unit normal,
$A=A[\iota,g]=-\frac{1}{2} \iota^*L_\nu g$
(for any extension of $\nu$)
and
$H=H[\iota,g]=\tr_{\iota^*g}A$
the corresponding second fundamental form
and mean curvature of $\Sigma$ in $(M,g)$,
and $A[\iota,\rho g]$ and $H[\iota, \rho g]$
the analogously defined second fundamental form
and mean curvature of $\Sigma$ in $(M, \rho g)$.

Given any $p \in M$ and unit vector $U \in T_pM$,
define $\Sigma$ to be the intersection of
a sufficiently small open
neighborhood of $p$ with the union of geodesics
through $p$ orthogonal to $U$,
so that $\Sigma$ is an embedded, two-sided hypersurface,
with $\nu$ the unit normal satisfying $\nu|_p = U$.
We write $\iota: \Sigma \to M$ for the inclusion map,
and for each $t \in \R$ we define the deformed inclusion
$\iota_t: \Sigma \to M$
by $\iota_t(q):=\exp^g_{\iota(q)} t\nu(q)$,
where $\exp^g$ is the exponential map for $(M,g)$.
Then $\iota_t$ is a $C^2_{loc}$ embedding
for sufficiently small $t$,
and we write $\nu(t)$
for the choice of unit normal for $\iota_t$
which is continuous in $t$ and agrees with $\nu$ at $t=0$.

Using (\ref{AHconf}) we find
  \begin{equation}
  \label{AH}
    \begin{aligned}
      &A[\iota, g]=0, \quad H[\iota, g]=0, \quad
      A[\iota, \rho g]
        =
        -\frac{1}{2}\rho|_{\Sigma}^{-1/2}
           \nu^c(\rho_{|c} \circ \iota)(\iota^*g), 
        \mbox{ and} \\
      &H[\iota_t, \rho g]
      =
      (\rho \circ \iota_t)^{-1/2}H[\iota_t, g]
        -\frac{n}{2}(\rho \circ \iota_t)^{-3/2}
         \nu^c(t)(\rho_{|c} \circ \iota_t),
    \end{aligned}
  \end{equation}
and differentiation then yields
\begin{equation}
\begin{aligned}
\label{Hdot1}
  \left.\frac{d}{dt}\right|_{t=0}H[\iota_t,\rho g]
  =
  &\rho|_\Sigma^{-1/2}\left.\frac{dH[\iota_t,g]}{dt}\right|_{t=0}
    +\frac{3n}{4}\rho|_\Sigma^{-5/2}\left[\nu^c(\rho_{|c} \circ \iota)\right]^2 \\
    &-\frac{n}{2}\rho|_\Sigma^{-3/2}\nu^c\nu^d(\rho_{|_g cd} \circ \iota).
\end{aligned}
\end{equation}
On the other hand from the standard expression
for the first variation of mean curvature
we have of course
$\left.\frac{d}{dt}\right|_{t=0} H[\iota_t, g]
  = R_{ab}[g]\nu^a\nu^b
$
and, using also the expression for $A[\iota, \rho g]$
in \eqref{AH} again,
\begin{equation}
\label{Hdot2}
  \begin{aligned}
    \left. \frac{d}{dt}\right|_{t=0} H[\iota_t,\rho g]
    =
    \Delta_{\rho|_\Sigma \iota^*g}\rho|_\Sigma^{1/2}
      +\frac{n}{4}\rho|_\Sigma^{-5/2}\left[\nu^c(\rho_{|c} \circ \iota)\right]^2
      +\rho|_\Sigma^{-1/2}R_{ab}[\rho g]\nu^a\nu^b.
  \end{aligned}
\end{equation}
Comparing, we obtain
\begin{equation}
\begin{aligned}
  R_{ab}[\rho g]\nu^a\nu^b
  =
  &R_{ab}[g]\nu^a\nu^b
     +\frac{n}{2}\rho|_\Sigma^{-2}[\nu^c(\rho_{|c} \circ \iota)]^2 \\
    &-\frac{n}{2}\rho|_\Sigma^{-1}\nu^c\nu^d(\rho_{|_g cd} \circ \iota)
    -\rho|_\Sigma^{1/2}\Delta_{\rho|_\Sigma \iota^*g}\rho|_\Sigma^{1/2}.
\end{aligned}
\end{equation}

From (\ref{Laplacianconf}) and (\ref{restrictedLaplacian}),
bearing in mind that $H[\iota,g]=0$,
  \begin{equation}
    \begin{aligned}
      \Delta_{\rho|_\Sigma \iota^*g}\rho^{1/2}
      =
      &\rho|_{\Sigma}^{-1}\left.\left(\Delta_g \rho^{1/2}\right)\right|_\Sigma
        -\rho|_{\Sigma}^{-1}D^2[g]\left(\rho^{1/2}\right)(\nu,\nu)|_\Sigma \\
        &+ \frac{n-2}{4}\rho|_\Sigma^{-5/2}\left.\abs{\nabla_g \rho}_g^2\right|_\Sigma
        -\frac{n-2}{4}\rho|_\Sigma^{-5/2}[\nu^c(\rho_{|c} \circ \iota)]^2.
    \end{aligned}
  \end{equation}
Since $U=\nu|_p$ was arbitrary
and $R_{ab}$ is symmetric,
by polarization we conclude,
after some simplification,
\begin{equation}
\begin{aligned}
  R_{ab}[\rho g]
  =
  &R_{ab}[g]+\frac{3(\dim M - 2)}{4}\rho^{-2}\rho_{|a}\rho_{|b}
    -\frac{\dim M - 2}{2}\rho^{-1}\rho_{|_g ab} \\
   &
    -\frac{1}{2}\rho^{-1}\left(\Delta_g \rho\right)g_{ab}
    -\frac{\dim M - 4}{4}\rho^{-2}\abs{\nabla_g \rho}_g^2g_{ab}.
\end{aligned}
\end{equation}
In particular, in three dimensions
\begin{equation}
\label{confric}
\begin{aligned}
  R_{ab}[M^3, \rho g]
  =
  &R_{ab}[M^3,g]+\frac{3}{4}\rho^{-2}\rho_{|a}\rho_{|b}
   -\frac{1}{2}\rho^{-1}\rho_{|_g ab} \\
   &-\frac{1}{2}\rho^{-1}\left(\Delta_g \rho\right)g_{ab}
   +\frac{1}{4}\rho^{-2}\abs{\nabla_g \rho}_g^2g_{ab}.
\end{aligned}
\end{equation}

\section{Mass and variation of mean curvature}
\label{massapp}
Let
$S$ and $M$ be smooth manifolds,
$\phi: S \to M$ a two-sided $C^2_{loc}$ codimension-one immersion,
and $\{g(t)\}_{t \in \R}$ a smooth one-parameter family of
$C^1_{loc}$ Riemannian metrics on $M$.
Pick a corresponding smooth one-parameter family
$\{\nu(t)\}$
of unit normals on $S$,
inducing corresponding
second fundamental form
$A_{\alpha\beta}(t)
 :=D[g(t)]_\alpha {\phi^c}_{,\beta}\nu^d(g_{cd}(t) \circ \phi)$
and mean curvature
$H(t):=(\phi^*g(t))^{\alpha\beta}A_{\alpha\beta}(t)$.
Set
$\gamma:=\phi^*g(0)$,
$A_{\alpha\beta}:=A_{\alpha\beta}(0)$,
$H:=H(0)$,
$\dot{g}:=\partial_t|_{t=0}g(t) \circ \phi$,
$\dot{\gamma}:=\partial_t|_{t=0} \phi^*\dot{g}(t)$,
and $\dot{H}:=\partial_t|_{t=0}H(t)$.
Then it is easy to compute
(as for example in Appendix B of \cites{Wbmoss})
that
  \begin{equation}
    \dot{H}
    =
    -A^{\alpha\beta}\dot{\gamma}_{\alpha\beta}
      +\frac{1}{2}H\dot{g}_{ab}\nu^a\nu^b
      +\frac{1}{2}\nu^c
        \left(\dot{g}_{c\sigma}^{\;\;\;\;|\sigma}
        +\dot{g}_{c\sigma}^{\;\;\;\;|\sigma}
        -\dot{g}^\sigma_{\;\;\sigma|c}\right),
  \end{equation}
where
vertical bars indicate differentiation relative to $g(0)$,
Roman (Greek) indices are raised and lowered
via $g(0)$ ($\gamma$)
and a Greek index on $\dot{g}$ indicates
(partial) pullback by $d\phi$, so that for example
$\dot{g}_{\alpha b}=\dot{g}_{ab}{\phi^a}_{,\alpha}$.

Of course
  \begin{equation}
    \begin{aligned}
      &\nu^c\dot{g}_{c\sigma}^{\;\;\;\;|\sigma}
      =
      \dot{g}_{cd}^{\;\;\;\;|d}\nu^c-\dot{g}_{ab|c}\nu^a\nu^b\nu^c, \\
      &\nu^c\dot{g}^\sigma_{\;\;\sigma|c}
      =
      \dot{g}^d_{\;\;d|c}-\dot{g}_{ab|c}\nu^a\nu^b\nu^c,
        \mbox{ and} \\
      &\nu^c\dot{g}_{c\sigma}^{\;\;\;\;|\sigma}
      =
      \left(\nu^c \dot{g}_c^{\;\;\sigma}\right)_{|\sigma}
        -A^{\rho\sigma}\dot{\gamma}_{\rho\sigma}
      =
      \left(\nu^c \dot{g}_c^{\;\;\sigma}\right)_{:\sigma}
        +H\dot{g}_{ab}\nu^a\nu^b-A^{\rho\sigma}\dot{\gamma}_{\rho\sigma},
    \end{aligned}
  \end{equation}
where colons indicate differentiation relative to $\gamma$.
It thus follows that
  \begin{equation}
    2\dot{H}
    =
    -A^{\alpha\beta}\dot{\gamma}_{\alpha\beta}
    +\nu^d\left(\dot{g}_{cd}^{\;\;\;\;|c}-\dot{g}^{c}_{\;\;c|d}\right)
    -\left(\nu^c\dot{g}_c^{\;\;\sigma}\right)_{:\sigma}.
  \end{equation}
If $S$ is closed, then
  \begin{equation}
  \label{C}
    \int_S
      \left(2\dot{H}+A^{\alpha\beta}\dot{\gamma}_{\alpha\beta}\right)
      \, \sqrt{\abs{\gamma}}
    =
    \int_S
      \left(\dot{g}_{cd}^{\;\;\;\;|c}-\dot{g}^{c}_{\;\;c|d}\right)
      \nu^d \, \sqrt{\abs{\gamma}}.
  \end{equation}

\section{Second variation of $\Bcal$}

\subsection*{Variation with respect to $f$ and $X$}
Given any $C^1$ map (not necessarily an immersion) $\varphi: P \to M$
from a smooth manifold $P$ into a smooth Riemannian manifold $(M,g)$
with corresponding Levi-Civita connection $D[TM,g]$,
we write $D[\varphi^*TM,g]$ for the unique connection on $\varphi^*TM$
satisfying the chain rule
  \begin{equation}
    D[\varphi^*TM]_\alpha\left(Z^c \circ \varphi\right)
    =
    \left(D[TM,g]_aZ^c {\varphi^a}_{,\alpha}\right) \circ \varphi.
  \end{equation}
Then $D[\varphi^*TM]$ is torsion-free and metric-compatible in the obvious senses.
We reserve the right to write simply $D$ in instances
when context suffices to identify the connection we have in mind.

Now let $\phi: S \to M$ be a smooth two-sided (codimension-one) immersion
of a smooth manifold $S$ into a complete smooth Riemannian manifold $(M,g)$
and write $\gamma_{\alpha\beta}:=\left(\phi^*g\right)_{\alpha \beta}$
for the corresponding induced metric on $S$.
Let $\nu \in \phi^*(TM)$ be a global unit normal for $\phi$
and write $A_{\alpha \beta}:=-D_{\alpha} \nu^c {\phi^d}_{,\beta} \left(g_{cd} \circ \phi\right)$
for the corresponding scalar-valued second fundamental form
and $H:=\gamma^{\alpha \beta}A_{\alpha \beta}$ for the corresponding
mean curvature.

Suppose also that $X^\alpha \in C^2(TS)$ and $f \in C^2(S)$
and define the vector fields $\xi^a, Z^a \in \phi^*(TM)$ by
  \begin{equation}
  \label{XxifZ}
    \xi := \phi_* X \qquad \mbox{and} \qquad Z:=\xi + f\nu.
  \end{equation}
In turn define the map the map $\Phi: S \times \R \to M$
and, for each $t \in \R$ the map $\phi[t]: S \to M$ by
  \begin{equation}
    \Phi(p,t):=\exp_{\phi(p)}^{(M,g)} tZ
    \qquad \mbox{and} \qquad
    \phi[t]:=\Phi(\cdot,t),
  \end{equation}
$\exp^{(M,g)}: TM \to M$ being the exponential map on $(M,g)$.
For $\abs{t}$ sufficiently small the map $\phi[t]$ is an immersion with continuous unit normal
$\nu[t] \in \phi[t]^*(TM)$ chosen so that $\nu[0]=\nu$
and so as to make continuous the vector field $N \in \Phi^*TM$
given by $N(p,t):=\nu[t](p)$.
We set $\gamma[t]:=\phi[t]^*g$,
$A[t]_{\alpha \beta}:=-D_{\alpha} \nu[t]^c {\phi[t]^d}_{,\beta} \left(g_{cd} \circ \phi[t]\right)$,
and $H[t]:=\gamma[t]^{\alpha\beta}A[t]_{\alpha\beta}$
(so that $\gamma[0]=\gamma$, $A[0]=A$, and $H[0]=H$).
We also set $\dot{\gamma}:=\partial_t \gamma[t]|_{t=0}$
and $\ddot{\gamma}:=\partial_t^2 \gamma[t]|_{t=0}$
and adopt like notation for $t$-derivatives of $A$ and $H$ at $t=0$.

Now we compute
  \begin{equation}
    \partial_t (g \circ \Phi)(\Phi_*V,\Phi_*W)
    =
    (g \circ \Phi)(D_V \Phi_*\partial_t, \Phi_*W) + (g \circ \Phi)(\Phi_*V,D_W \Phi_*\partial_t)
  \end{equation}
and in turn
  \begin{equation}
    \partial_t^2 (g \circ \Phi)(\Phi_*V,\Phi_*W)
    =
    2(g \circ \Phi)((R \circ \phi)(\Phi_*\partial_t,\Phi_*V)\Phi_*\partial_t,\Phi_*W)
      +2(g \circ \Phi)(D_V\Phi_*\partial_t, D_W \Phi_*\partial_t),
  \end{equation}
$R$ being the Riemann curvature of $(M,g)$, our conventions specified by
\eqref{curvatureconvention}.
It follows that
  \begin{equation}
  \label{gammavar}
    \dot{\gamma}_{\alpha \beta}
    =
    Z_{c|\alpha} {\phi^c}_{,\beta}+Z_{c|\beta} {\phi^c}_{,\alpha}
    \quad \mbox{and} \quad
    \ddot{\gamma}_{\alpha \beta}
    =
    2{Z^c}_{|\alpha}Z_{c|\beta}
      +2\left(R_{abcd} \circ \phi\right)Z^aZ^c {\phi^b}_{,\alpha}{\phi^d}_{,\beta},
  \end{equation}
where the bar $|$ indicates differentiation via $D[\phi^*TM,g]$ and Roman indices are raised and lowered via $g \circ \Phi$.

Next, using the product rule, we find
  \begin{equation}
  \label{normal}
    \begin{aligned}
      &(g \circ \Phi)\left(D_{\partial_t}N,N\right)=0, \\
      &(g \circ \Phi)\left(D_{\partial_t}N,\Phi_*V\right)|_{t=0}
      =
      -Z_{c|\alpha}\nu^c V^\alpha, \mbox{ and} \\
      &(g \circ \Phi)\left(D_{\partial_t}D_{\partial_t}N,N\right)|_{t=0}
      =
      -(g \circ \Phi)\left(D_{\partial_t}N,D_{\partial_t}N\right)|_{t=0}
        =-Z_{c|\rho}{Z_d}^{|\rho}\nu^c \nu^d,
    \end{aligned}
  \end{equation}
where Greek indices are raised and lowered via $\gamma$.
We also compute
  \begin{equation}
  \label{firstA}
    \begin{aligned}
      \partial_t (g \circ \Phi)\left(D_V \Phi_*W,N\right)
      &=
      (g \circ \Phi)\left((R \circ \Phi)\left(\Phi_*\partial_t,\Phi_*V\right)\Phi_*W,N\right) \\
        &\;\;\;\; +(g \circ \Phi)\left(D_V D_W \Phi_*\partial_t,N\right)
        +(g \circ \Phi)\left(D_V \Phi_*W, D_{\partial_t}N\right)
    \end{aligned}
  \end{equation}
and, at this point specializing to the flat case $R_{abcd}=0$,
  \begin{equation}
  \label{secondA}
      \partial_t^2(g \circ \Phi)\left(D_V \Phi_*W,N\right)
      =
      2(g \circ \Phi)\left(D_V D_W \Phi_*\partial_t,D_{\partial_t}N\right)
        +(g \circ \Phi)\left(D_V \Phi_*W,D_{\partial_t} D_{\partial_t}N\right).
  \end{equation}
Applying \eqref{normal} to \eqref{firstA} and \eqref{secondA},
we obtain
  \begin{equation}
  \label{Avar}
    \dot{A}_{\alpha \beta}=Z_{c|\alpha \beta}\nu^c
    \qquad \mbox{and} \qquad
    \ddot{A}_{\alpha \beta}
    =
    -2Z_{c|\alpha \beta}{Z_d}^{|\sigma}{\phi^c}_{,\sigma}\nu^d
      -A_{\alpha \beta}Z_{c|\rho}{Z_d}^{|\rho}\nu^c \nu^d.
  \end{equation}

From \eqref{XxifZ}
  \begin{equation}
    \begin{aligned}
      {Z^c}_{|\sigma}
      =
      &{X^\rho}_{:\sigma}{\phi^c}_{,\rho}+A_{\rho \sigma}X^\rho \nu^c
        +f_{,\sigma}\nu^c - f{A^\rho}_{\sigma}{\phi^c}_{,\rho}, \quad \mbox{and so} \\
      {Z^c}_{|\alpha\beta}
      =
      &{X^\rho}_{:\alpha\beta}{\phi^c}_{,\rho}
        +{X^\rho}_{:\alpha}A_{\rho \beta}\nu^c
        +A_{\rho \alpha:\beta}X^\rho\nu^c
        +A_{\rho\alpha}{X^\rho}_{:\beta}\nu^c
        -A_{\rho\alpha}{X^\rho}{A^\sigma}_\beta{\phi^c}_{,\sigma} \\
        &+f_{:\alpha\beta}\nu^c
        -f_{,\alpha}{A^\sigma}_\beta{\phi^c}_{,\sigma}
        -f_{,\beta}{A^\sigma}_\alpha{\phi^c}_{,\sigma}
        -f{A^\rho}_{\alpha:\beta}{\phi^c}_,{\rho}
        -f{A^\rho}_\alpha A_{\rho\beta}\nu^c.
    \end{aligned}
  \end{equation}
Specializing further to the case
when $\phi=\iota: S \to M$ is the inclusion map
of the standard unit sphere $S=\Sph^2$ in $M=\R^3$,
we have $A=-\gamma$, whence
  \begin{equation}
  \label{specialization}
    \begin{aligned}
      &{Z^c}_{|\sigma}
      =
      {X^\rho}_{:\sigma}{\phi^c}_{,\rho} + f{\phi^c}_{,\sigma}
        +\left(f_{,\sigma}-X_\sigma\right)\nu^c \quad \mbox{and} \\
      &{Z^c}_{|\alpha\beta}
      =
      {X^\rho}_{:\alpha\beta}{\phi^c}_{,\rho}
        -X_\alpha{\phi^c}_{,\beta}
        +f_{,\alpha}{\phi^c}_{,\beta}
        +f_{,\beta}{\phi^c}_{,\alpha}
        +\left(
          f_{:\alpha\beta}
          -f\gamma_{\alpha\beta}
          -X_{\alpha:\beta}
          -X_{\beta:\alpha}
        \right)\nu^c.
    \end{aligned}
  \end{equation}

Applying these last expressions in \eqref{gammavar} and \eqref{Avar},
we conclude
  \begin{equation}
    \begin{aligned}
      &\dot{\gamma}_{\alpha\beta}
      =
      2f\gamma_{\alpha\beta}+X_{\alpha:\beta}+X_{\beta:\alpha}, \\
      &\dot{A}_{\alpha\beta}
      =
      f_{:\alpha\beta}-f\gamma_{\alpha\beta}-X_{\alpha:\beta}-X_{\beta:\alpha}, \\
      &\ddot{\gamma}_{\alpha\beta}
      =
      2{X^\sigma}_{:\alpha}X_{\sigma:\beta}
      +2f^2\gamma_{\alpha\beta}
      +2f_{,\alpha}f_{,\beta}
      +2X_\alpha X_\beta \\
      &\;\;\;\;\;\;\;\;\;
      -2f_{,\alpha}X_\beta - 2f_{,\beta}X_\alpha
      +2fX_{\alpha:\beta}+2fX_{\beta:\alpha}, \mbox{ and} \\
      &\ddot{A}_{\alpha\beta}
      =
      \left(\abs{X}^2\right)_{:\alpha\beta}
      -2{X^\sigma}_{:\alpha}X_{\sigma:\beta}
      -2X_\alpha X_\beta
        -2f^{:\sigma}X_{\sigma:(\alpha\beta)}
        +6X_{(\alpha}f_{,\beta)}
        -4f_{,\alpha}f_{,\beta}  \\
        &\;\;\;\;\;\;\;\;\;
          +\abs{df}^2\gamma_{\alpha\beta}+\abs{X}^2\gamma_{\alpha\beta}
          -2(Xf)\gamma_{\alpha\beta},
    \end{aligned}
  \end{equation}
whence it follows,
setting
  \begin{equation}
    (KX)_{\alpha\beta}:=X_{\alpha:\beta}+X_{\beta:\alpha}=2X_{(\alpha:\beta)}
  \end{equation}
that
  \begin{equation}
  \label{2ndvarsummary}
    \begin{aligned}
      \dot{\gamma}^\sigma_{\;\;\sigma}
      &=
      4f + 2 \div X, \\
      \ddot{\gamma}^\sigma_{\;\;\sigma}
      &=
      2\abs{DX}^2+4f^2+2\abs{df}^2+2\abs{X}^2-4Xf+4f \div X, \\
      \dot{\gamma}_{\alpha\beta}\dot{\gamma}^{\alpha\beta}
      &=
      8f^2+8f \div X+\abs{KX}^2, \\
      \dot{\gamma}^{\alpha\beta}\dot{A}_{\alpha\beta}
      &=
      2f\Delta f-4f^2-6f \div X+2X_{\alpha:\beta}f^{\alpha:\beta}-\abs{KX}^2, \\
      \ddot{A}^\sigma_{\;\;\sigma}
      &=
      \Delta \abs{X}^2-2\abs{DX}^2-2\abs{df}^2+2Xf
      +2X^{\alpha:\beta}f_{:\alpha\beta}
      -2(X^{\alpha:\beta}f_{,\alpha})_{:\beta},  \\
      \dot{H}
      &=
      \Delta f + 2f, \quad \mbox{and} \\
      \ddot{H}
      &=
      (\Delta+2)\abs{X}^2-4f(\Delta+1)f
      -2Xf-2X_{\alpha:\beta}f^{:\alpha\beta}
      -2(X^{\alpha:\beta}f_{,\alpha})_{:\beta}.
    \end{aligned}
  \end{equation}

\subsection*{Total variation}
Now, given $f,v \in C^2(\partial M)$, $X \in C^2(TM)$, and $t \in \R$
we define $\rho_t: M \to \R$ and $\phi_t: M \to \R$ by
  \begin{equation}
  	\rho_t:=1+tv
  	\qquad \mbox{and} \qquad
  	\phi_t(\vec{x})
      :=
      \vec{x}
        +t\psi\left(\abs{\vec{x}}\right)
          \left[
            f\left(\frac{\vec{x}}{\abs{\vec{x}}}\right)+\vec{\xi}\left(\vec{x}\right)
          \right].
  \end{equation}
We have just computed, as summarized in \eqref{2ndvarsummary},
the first and second derivatives at $t=0$ of
$\left(\iota^*\phi_t^*\delta\right)^\sigma_{\;\;\sigma}$
and $\Hcal[\iota,\phi_t^*\delta]$.
Making use of these calculations
we next compute
  \begin{equation}
  \label{t2g}
    \begin{aligned}
   	  \left.\frac{d^2}{dt^2}\right|_{t=0} 
      \left(\iota^*\rho_t \phi_t^*\delta\right)^\sigma_{\;\;\sigma}
  	  =
  	  &8vf + 4 v\div X + 2\abs{DX}^2+4f^2 \\
  	    &+2\abs{df}^2+2\abs{X}^2-4Xf+4f \div X
  	\end{aligned}
  \end{equation}
and,
using also the identity
  \begin{equation}
  \label{t2H}
    \Hcal[\iota,\rho_t \phi_t^*\delta]
     =
     \left(\iota^*\rho_t^{-1/2}\right)
       \Hcal[\iota,\phi_t^*\delta]
       -\iota^*\rho_t^{-3/2}N[\phi_t^*\delta]\rho_t
    \end{equation}
 (see for example (B.27) in \cites{Wbmoss})
 along with \eqref{normal} and \eqref{specialization},
  \begin{equation}
    \begin{aligned}
      \left.\frac{d^2}{dt^2}\right|_{t=0} \Hcal[\iota,\rho_t\phi_t^*\delta]
      =
      &(\Delta+2)\abs{X}^2-4f(\Delta+1)f-v(\Delta+2)f-2X(f+v) \\
        &+3vv_{,r}-\frac{3}{2}v^2+2f^{:\alpha}v_{,\alpha}
        -2X_{\alpha:\beta}f^{:\alpha\beta}
        -2\left(X^{\alpha:\beta}f_{,\alpha}\right)_{:\beta}.
    \end{aligned}
  \end{equation}

\section{Divergence and Ricci deformation on Euclidean domains}
\label{RicciApp}
Assume
$n \in \Z \cap [3, \infty)$,
$k \in \Z \cap [0, \infty)$,
$\alpha \in (0,1)$,
and $\beta \in (0, n-2) \backslash \Z$.
Let $\Omega$ be a (not necessarily bounded)
connected, open subset of $\R^n$
with smoothly embedded, compact
(but not necessarily connected) boundary.
For any symmetric tensor $F_{ab}$ on $\Omega$
we set
$\widehat{F}_{ab} := F_{ab} - \frac{1}{2}{F^c}_c\delta_{ab}$
and
$\widecheck{F} := F_{ab} - \frac{1}{n-2}{F^c}_c\delta_{ab}$,
so that $\widecheck{\widehat{F}}=\widehat{\widecheck{F}}=F$,
and
we call $F$ \emph{self-equilibrated on $\Omega$}
if for each closed embedded hypersurface $S \subset \Omega$
and each Killing field $X$ of $\R^n$
we have $\int_S F_{ab}X^a N^b$=0,
where $N$ is a global unit normal for $S$.
If $F$ is $C^1$ in $\Omega$ and continuous up to the boundary,
$F$ is then self-equilibrated on $\Omega$
if and only if
${F_{ab}}^{|b}=0$ on $\Omega$
and $\int_S F_{ab}X^a N^b = 0$
for each Killing field $X$
and each component $S$ of $\partial \Omega$.

Writing $\dot{R}_{ab}[\delta]$
as above for the linearization at $\delta$
of the Ricci operator,
note that
$\widehat{F}_{ab}$
is self-equilibrated
on $\Omega$ for 
$F_{ab} := \dot{R}_{ab}[\delta]h$ 
with $h_{ab}$ any given $C^3$ symmetric tensor
on $\Omega$
(by virtue of the Bianchi identity applied to
$\dot{R}_{ab}[\delta]\widetilde{h}$ for any compactly supported
$C^3$ symmetric extension $\widetilde{h}_{ab}$ of $h$
to all $\R^n$).
In fact
self-equilibration of a symmetric tensor
(assuming appropriate decay if $\Omega$ is unbounded)
is also a sufficient condition
for it to be the linearized Einstein tensor
of some deformation of $\delta$ on $\Omega$,
as shown below.
(Sufficiency holds also in dimension 1, trivially,
but fails in dimension 2,
where every
$\dot{R}[\delta]_{ab}h$
is scalar
and yet on the unit disc,
for example,
there is a two-dimensional space of constant 
trace-free symmetric tensors, all of which
are self-equilibrated.)

We will solve the problem
$\dot{R}_{ab}[\delta]h = F_{ab}$
by first extending $F$ to a self-equilibrated tensor
on $\R^n$,
but the extension we specify loses regularity,
so we first solve the problem modulo error smoother
than $F$.
Throughout we will make use of the identities
\begin{align}
\label{KBminusDelta}
  \dot{R}_{ab}[\delta]h
  =\frac{1}{2}\widehat{h}_{c(a\;\;|b)}^{\;\;\;\;|c}
    -\frac{1}{2}{h_{ab|c}}^{|c}
  \quad \mbox{and} \\
\label{foliateddiv}
\begin{split}
  \div (T + V \otimes N + N \otimes V + \xi N \otimes N)
  =
  &\left[\div^\top T + (\overline{D}_N + S - H)V)\right] \\
   & + \left[\div^\top V + (N-H)\xi
            + A_{\alpha\beta}T^{\alpha\beta}\right]N,
\end{split}
\end{align}
where we assume (locally) a foliation by hypersurfaces
with $N^a$ unit and normal to each leaf,
$V^a$ and $T^{ab}=T^{ba}$ purely tangential to leaves,
$\div^\top$ the intrinsic (vector or symmetric tensor)
divergence on each leaf, $\overline{D}$ the ambient
connection, and $S$ and $H$ the shape operator
and mean curvature respectively of the leaves.

\begin{lemma}
[Right parametrix for linearized Ricci]
\label{smoothing}
There exists a constant
$C=C(\alpha, \beta, k, \Omega) > 0$
such that if
$F^{ab} \in C^{k+1, \alpha, 2+\beta}(T\Omega^{\odot 2})$
is self-equilibrated on $\Omega$,
then there exist
$h_{ab}$, $G_{ab}$ on $\Omega$ such that
$\dot{R}_{ab}[\delta]h =
  \widecheck{F}_{ab} + \widecheck{G}_{ab}$,
$\norm{h}_{k+3,\alpha,\beta}
  +\norm{G}_{k+2,\alpha,2+\beta}
  \leq C\norm{F}_{k+1,\alpha,2+\beta}$,
and $G$ is self-equilibrated on $\Omega$.
\end{lemma}

\begin{proof}
First let $\eta_{ab}$ be the solution
to the Poisson equation
${\eta_{ab|c}}^{|c}=-2\widecheck{F}_{ab}$
with $\eta|_{\partial \Omega}=0$.
In light of \eqref{KBminusDelta}
any symmetric tensor $\theta_{ab}$
satisfying ${\theta_{ab|c}}^{|c}=0$
and
${\theta_{ac}}^{|c}
  =\zeta_a:=-\widehat{\eta}_{ac}^{\;\;\;|c}$
yields an exact solution to
$\dot{R}_{ab}[\delta](\eta+\widecheck{\theta})
 =\widecheck{F}_{ab}$.
It is easy though to construct instead a
harmonic, symmetric tensor $\vartheta_{ab}$
satisfying the divergence condition to first order,
as follows.
Writing $\mathcal{P}$ for the operator
taking tensors on $\partial \Omega$
to their bounded (Cartesian componentwise)
harmonic extensions on $\Omega$
and $N$ for the inward unit normal on $\partial \Omega$,
let $X^a=V^a + \xi N^a$ (with $V \perp N$) and
$\widetilde{\xi}$
be the solutions on $\partial \Omega$ to
$(\overline{D}_N - 1)\mathcal{P}X
 = \zeta - \zeta_cN^cN$
and 
$(N-1)\mathcal{P}\widetilde{\xi}
 = \zeta^cN_c - \div^\top V - 2N\xi$.
Then,
referring to \eqref{foliateddiv}
(and using the regularity of $\partial \Omega$),
on $\partial \Omega$ the symmetric tensor
$(\mathcal{P}X) \otimes N + N \otimes (\mathcal{P}X)
 + \left(\mathcal{P}\widetilde{\xi}\right) N \otimes N$
has divergence $\zeta$ plus a $C^{k+3, \alpha}$ function,
and consequently
(since
$[\mathcal{P}(\cdot|_{\partial \Omega}),
  \cdot \otimes Z]$ has order $-1$
  for any smooth tensor $Z$)
so does
$\vartheta := \mathcal{P}(X \otimes N + N \otimes X
                           + \widetilde{\xi} N \otimes N)$.
Finally note that, by its definition above,
$\zeta$ is harmonic,
as of course is $\vartheta$,
so $h := \eta + \widecheck{\vartheta}$
satisfies $\div \widehat{h} \in C^{k+3,\alpha}(\Omega)$.
The proof is now completed by taking
$\widecheck{G}$ to be the first term in \eqref{KBminusDelta},
its self-equilibration following from that of $F$
and $(\dot{R}_{ab}[\delta]h)\widehat{\hphantom{a}}$;
the estimates are clear from the construction.
\end{proof}

\begin{lemma}
[Right inverse for divergence
with support contained in a hypercube]
\label{convexdivergence}
Suppose $d \in \Z \cap [1, \infty)$
and set $Q := [-1,1]^d \subset \R^d$.
There exists a constant $C=C(d, k, \alpha)>0$
such that the following hold.
\begin{enumerate}[(i)]
  \item If $f \in C^{k+1, \alpha}(\R^d)$
        has support
        contained in $Q$
        and $\int_{\R^d} f = 0$,
        then there exists
        $X^a \in C^{k+1, \alpha}(T\R^d)$
        with support contained in $Q$,
        ${X^c}_{|c}=f$,
        and $\norm{X}_{k+1,\alpha}
              \leq C\norm{f}_{k+1, \alpha}$.
  \item If $F^a \in C^{k+1, \alpha}(T\R^d)$
        has support contained in $Q$
        and $\int_{\R^d} F^cX_c = 0$
        for every Killing field $X^a$ on $\R^d$,
        then there exists
        $S^{ab} \in C^{k+1, \alpha}(T\R^d \odot T\R^d)$
        with support contained in $Q$,
        ${S_{ab}}^{|b} = F_a$,
        and $\norm{S}_{k+1, \alpha}
             \leq C\norm{F}_{k+1, \alpha}$.
  \item If $F^a \in C^{k+1, \alpha}(T\R^d)$
        has support contained in $Q$
        and $B$ is any open ball contained
        in the interior of $Q$,
        then there exists a 
        smooth vector field
        $\rho^a$ such that 
        $\rho$ has support contained in $B$
        and there exists
        $S^{ab} \in C^{k+1, \alpha}(T\R^d \odot T\R^d)$
        with support contained in $Q$,
        ${S_{ab}}^{|b} = F_a + \rho_a$,
        and $\norm{S}_{k+1, \alpha}
             \leq C\norm{F}_{k+1, \alpha}$.
\end{enumerate}
\end{lemma}

\begin{proof}
We prove (i) and (ii) by induction on $d$,
the case $d=1$ following immediately for both
by integration.
Assume $d \geq 2$.
For the inductive step we will apply
\eqref{foliateddiv} and the analogous
decomposition
$\div (V + \xi N) = \div^\top V + (N-H)\xi$,
in the same notation, for vector fields
to the foliation of $\R^n$
by hyperplanes orthogonal to the $x^d$-axis.
To simplify the notation slightly
we set $t:=x^d$
and $x:=\left(x^1, \cdots, x^{d-1}\right)$.
Fix a $\psi \in C_c^\infty(\R)$
with support contained in $[-1/2, 1/2]$
and $\int \psi = 1$.

For (i) define functions $f_\perp$ and $f_\top$
on $\R^d$ by
$f_\top(x, t) := \psi(t)\int_{\R} f(x,\tau) \, d\tau$
and $f_\perp := f - f_\top$,
so that both $f_\perp$ and $f_\top$
have support contained in $Q$,
$\int_{\R} f_\perp(x,t) \, dt = 0$
for each $x \in \R^{d-1}$
and $\int_{\R^{d-1}} f_\top(x,t) \, \abs{dx} = 0$
(since $\int_{\R^d} f = 0$) for each $t \in \R$.
Now we define $\xi$ on $\R^d$
by
$\xi(x,t) := \int_{-\infty}^t f_\perp(x, \tau) \, d\tau$,
and by the inductive hypothesis there exists
a vector field $V$ on $\R^d$ with values
orthogonal to $N=\partial_t=\partial_d$ and satisfying
$\div V(x,t) = f_\top(x,t)$ for each $t \in \R$.
Taking $X := V + \xi N$ concludes the proof of (i).

For (ii) we first reduce to the case that
$F$ is everywhere orthogonal to $N$, as follows.
Set $f := F^cN_c$, and define $V$ and $\xi$
exactly as in the previous paragraph,
with $\int_{\R^{d-1}} f_\top(x,t) \, \abs{dx} = 0$
now because $\int_{\R^d} f = \int_{\R^d} F^cN_c = 0$,
$N=\partial_t=\partial_d$ being a Killing field.
Then
$\div (V \otimes N + N \otimes V + \xi N \otimes N)$
is orthogonal to the Killing fields
and, 
referring to \eqref{foliateddiv},
equals $\overline{D}_N V + fN$,
where the first term is orthogonal to $N$.

Thus we now assume
$F$ is purely tangential to the leaves.
Fix a $\phi \in C_c^\infty([0,\infty))$
with support contained in $[0,1]$,
$\phi \geq 0$, and
$\int_{\R^{d-1}} \phi(\abs{x}) \, d\abs{x} = 1$.
For each $i \in \Z \cap [1,d-1]$
define the vector fields
${}_iK := \partial_i$
and
${}_i\widetilde{K} := \phi(\abs{x}) \; {}_iK$
and,
setting
$a := 
 1/\int_{\R^{d-1}} 2(x^1)^2\phi(\abs{x}) \, d\abs{x}$,
for each $i<j \in \Z \cap [1,d-1]$
define the vector fields
${}_{i,j}K := x^i\partial_j - x^j\partial_i$
and
${}_{i,j}\widetilde{K} :=
 a\phi(\abs{x}) \; {}_{i,j} K
$.
The ${}_iK$ and ${}_{i,j}K$ form a basis
for the Killing fields on $\R^{d-1}$,
and
the ${}_i\widetilde{K}$ and ${}_{i,j}\widetilde{K}$
restrict (acting by inner product and integration)
to the dual basis.
Note that each $K_{i,j}$ is divergence-free,
despite the cutoff.
 
Next set
\begin{equation}
\begin{aligned}
  &F_\rightarrow(x,t) :=
    \sum_{i=1}^{d-1}
    \left(\int_{\R^{d-1}} {}_iK^cF_c(x,t) \, d\abs{x}\right)
    \; {}_i\widetilde{K}(x,t), \\
  &F_\circ(x,t) :=
    \sum_{1 \leq i < j \leq d-1}
    \left(
      \int_{\R^{d-1}} {}_{i,j}K^cF_c(x,t) \, d\abs{x}
    \right)
    \; {}_{i,j}\widetilde{K}(x,t), \quad \mbox{and} \\
  &F_\top := F - F_\rightarrow - F_\circ,
\end{aligned}
\end{equation}
so that $F_\rightarrow$, $F_\circ$, and $F_\top$
all have support contained in $Q$,
$F_\top(\cdot, t)$ is orthogonal to the Killing
fields on $\R^{d-1}$ for each $t \in \R$,
and
$\int_{\R} F_\circ(x,t) \, dt
 = \int_{\R} F_\rightarrow(x,t) \, dt = 0$
for each $x \in \R^{d-1}$.
In particular, by the inductive hypothesis,
there exists
a symmetric tensor $T_{ab}$ with $T_{ab}N^b=0$
everywhere,
support contained in $Q$,
and $\div^\top T = F_\top$.
Furthermore
$V_\circ(x,t) := \int_{-\infty}^t F_\circ(x, \tau) \, dt$
has support contained in $Q$
and $\div^\top V_\circ(x,t) = 0$ for each $t \in \R$
(since each ${}_{i,j}\widetilde{K}$
has vanishing divergence, as noted above).
Referring again to \eqref{foliateddiv},
we therefore have
$\div (T+V_\circ \otimes N + N \otimes V_\circ)
 = F_\top + F_\circ$.

Finally we will use the orthogonality
of $F$
to
$t\partial_i - x^i\partial_t
 =x^d\partial_i -x^i\partial_d$
for each $i \in \Z \cap [1,d-1]$.
Setting
$c_i(t) := \int_{\R^{d-1}} {}_iK^cF_c(x,t) \, d\abs{x}$,
so that
$F_\rightarrow(x,t) = \sum_{i=1}^{d-1}
 c_i(t) \, {}_i\widetilde{K}(x,t)$,
we then have
$\int_{\R} tc_i(t) \, dt = 0$ for each $i$,
but orthogonality of $F$ to each ${}_iK$
gives $\int_{\R} c_i(t) \, dt = 0$,
so
$\int_{\R} \int_{-\infty}^t c_i(\tau) \, d\tau \, dt = 0$
too.
It follows not only that
$V_\rightarrow(x,t)
 :=\int_{-\infty}^t F_\rightarrow(x,\tau) \, d\tau
 $
 has support contained in $Q$
 but that
 $\xi_\rightarrow(x,t)
  :=
  -\int_{-\infty}^t
    \div^\top V_\rightarrow(x,\tau) \, d\tau
$
does too.
Since
$\div (V_\rightarrow \otimes N
       + N \otimes V_\rightarrow 
       + \xi_\rightarrow N \otimes N)
  = F_\rightarrow$,
the proof of (ii) is complete,
the estimates being clear from the construction.
For (iii) we construct $\rho$
by cutting off a basis for the Killing fields
on $\R^d$,
as done for $\R^{d-1}$ above,
but ``centered'' on $p$,
and choosing coefficients so that
$F^a - \rho^a$ is orthogonal to the Killing fields;
then we apply (ii).
\end{proof}

For alternative approaches
to constructing compactly supported symmetric tensors
of prescribed divergence 
see \cites{Delay, IO}
(though neither states results
in a form we can directly apply to our ends here.)

\begin{lemma}[Rough extension of self-equilibrated fields]
\label{roughextension}
There exists a constant $C=C(\alpha, k, \Omega) > 0$
such that every self-equilibrated
$F^{ab} \in C^{k+2,\alpha}(T\Omega^{\odot 2})$
admits a self-equilibrated
extension
$\overline{F}^{ab} \in C^{k+1, \alpha}(T\R^n \odot T\R^n)$
such that
$\overline{F}|_{\R^n \backslash \Omega}$
has compact support
and
$\norm{\overline{F}|_{\R^n \backslash \Omega}}_{k+1, \alpha}
  \leq C\norm{F|_{\partial \Omega_1}}_{k+2,\alpha}$,
where $\partial \Omega_1$ is the set of all points
in $\Omega$
at distance at most $1$ from $\partial \Omega$.
\end{lemma}

\begin{proof}
For any $G \subset \R^n$ write $\overline{G}$
for its closure
and for any $E \subset \partial \Omega$
and $\epsilon>0$
write $E^\epsilon$ for the set of all points
in $\R^n \backslash \partial \Omega$
having distance less than $\epsilon$ from $E$.
Since $\partial \Omega$
is $C^\infty$, compact, and embedded
there exist
$\epsilon>0$
and open sets
$Q_1, \cdots Q_N$
in $\partial \Omega$ (with its induced metric)
covering $\partial \Omega$
such that for each $Q_i$
we have that
$\overline{Q_i^\epsilon} \cap \partial \Omega
 =\overline{Q_i}$
and moreover that
item (iii) of Lemma \ref{convexdivergence}
holds with the hypercube $Q$ replaced by
$\overline{Q_i^\epsilon}$.
(To verify that this last part of the claim
 can be achieved
 note that at each $p \in \partial \Omega$
 we can find a local coordinate system
 $\Phi: \mathcal{U} \ni p \to \R^n$
 and an open neighborhood $Q_p \subset \partial \Omega$
 of $p$ such that
 $\Phi(\overline{Q_p^\epsilon})$
 is a hypercube of edge length $\epsilon$,
and for any $\ell>0$,
after rescaling so that $\Phi(Q_p)$ has unit size,
we can make 
$\Phi^*\delta$ arbitrarily $C^\ell$-close to $\delta$
on $\mathcal{U}$
by taking $\epsilon$ small.
For $\epsilon$ small enough \ref{convexdivergence}.(iii)
can then be applied iteratively
to produce an exact solution.)

Now we extend $F$ (continuously in $F$)
to some
$\widetilde{F}^{ab} \in C^{k+2,\alpha}(T\R^n \odot T\R^n)$
whose restriction to $\R^n \backslash \Omega$
has support contained in $\partial \Omega^\epsilon$.
Using a partition of unity on $\partial \Omega$
subordinate to $\{Q_i\}_{i=1}^N$
and extended constantly in the normal direction,
we obtain the decomposition
$\widetilde{F}_{ab}^{\;\;|b}=\sum_{i=1}^N {}_iF_a$,
where each ${}_iF_a$ has support contained
in $\overline{Q_i^\epsilon}$.
Then, by \ref{convexdivergence}.(iii)
with $\overline{Q_i^\epsilon}$ in place of $Q$,
for each ${}_iF_a$
there exist
a smooth vector field ${}_i\rho^a$
with support contained in the interior $Q_i^\epsilon$
and ${}_iS_{ab} \in C^{k+1,\alpha}(T\R^n \odot T\R^n)$
with support contained in $\overline{Q_i^\epsilon}$
and satisfying
${}_i{S_{ab}}^{|b}={}_iF_a - {}_i\rho_a$.
Note that
$\int_R^n {}_i\rho^cX_c = \int_\R^n {}_iF^cX_c$
for every Killing field $X$ on $\R^n$.
Summing,
we obtain $S_{ab}$ such that
$\widetilde{F} - S \in C^{k+1,\alpha}(T\R^n \odot T\R^n)$
has divergence $\sum_{i=1}^N {}_i\rho$
and its restriction to $\R^n \backslash \Omega$
has support contained in $\partial \Omega^\epsilon$.

By applying Lemma \ref{convexdivergence}
repeatedly we can ``move'' each ${}_i\rho$
freely within the component of $\partial \Omega^\epsilon$
that contains it
by adding to $\widetilde{F} - S$
$C^{k+2}$ symmetric tensor fields with support
contained in $\partial \Omega^\epsilon$,
so that the resulting sum has divergence
$\sum_{i=1}^N {}_i\widetilde{\rho}$,
where each ${}_i\widetilde{\rho}$
is smooth and can be chosen to have support
contained in any open ball in any hypercube
contained in $\partial \Omega^\epsilon$
and containing the support of ${}_i\rho$
and with ${}_i\widetilde{\rho}-{}_i\rho$
orthogonal to the Killing fields.
In this way we can consolidate all the ${}_i\rho$
(iteratively replacing each by
${}_i\widetilde{\rho}$ as above)
for a given component of $\partial \Omega$
into a single hypercube
contained in $\partial \Omega^\epsilon$,
but by the construction of the ${}_i\rho$
and the self-equilibration assumption
on the original $F$,
the sum of the ${}_i\rho$ contained
in a given component of $\partial \Omega^\epsilon$
is orthogonal to all Killing fields.
Therefore we can apply Lemma \ref{convexdivergence}.(ii)
(once for each component of $\partial \Omega$)
to eliminate all the ${}_i\rho$ by adding further
$C^{k+2}$ symmetric tensor
fields with support contained
in $\partial \Omega^\epsilon$.
Writing $\widetilde{S}_{ab}$
for the sum of all symmetric tensors
introduced in this paragraph to consolidate
and eliminate the ${}_i\rho$,
we conclude by taking
$\overline{F}_{ab}
 :=\widetilde{F}_{ab}-S_{ab}+\widetilde{S}_{ab}$.
\end{proof}

\begin{prop}[Right inverse for linearized Ricci]
\label{prescribingricci}
There exists a constant
$C=C(\alpha, \beta, k, \Omega) > 0$
such that if
$F^{ab} \in C^{k+1, \alpha, 2+\beta}(T\Omega^{\odot 2})$
is self-equilibrated on $\Omega$,
then there exists a
symmetric tensor
$h_{ab}$ on $\Omega$ such that
$\dot{R}_{ab}[\delta]h = \widecheck{F}_{ab}$,
$\widehat{h}$ is self-equilibrated,
and
$\norm{h}_{k+3,\alpha,\beta}
  \leq C\norm{F}_{k+1,\alpha,2+\beta}$.
\end{prop}

\begin{proof}
By Lemma \ref{smoothing}
we  may assume that
$F \in C^{k+2, \alpha, 2+\beta}(T\Omega^{\odot 2})$.
By Lemma \ref{roughextension}
$F$ has a self-equilibrated extension
$\overline{F} \in
  C^{k+1, \alpha, 2+\beta}(T\R^n \odot T\R^n)$.
Let $h^{ab}$ be the unique solution in 
$C^{k+3, \alpha, \beta}(T\R^n \odot T\R^n)$
to ${h_{ab|c}}^{|c}=-2\widecheck{\overline{F}}_{ab}$
on $\R^n$.
Then $\widehat{h}_{ab\;\;\;|c}^{\;\;\;|bc}=0$
on $\R^n$,
so, by uniqueness of solutions in $C^{k+2, \alpha, \beta}$
to this Poisson equation,
$\widehat{h}$ is self-equilibrated on $\R^n$,
and therefore, in light of \eqref{KBminusDelta},
satisfies
$\dot{R}_{ab}[\delta]h=\widecheck{\overline{F}}_{ab}$
on $\R^n$. Restriction to $\Omega$ completes the proof.
\end{proof}

Conversely, we can use Proposition \ref{prescribingricci}
to solve the extension problem.
\begin{cor}[Smooth extension of self-equilibrated fields]
\label{smoothextension}
There exists a constant $C=C(\alpha, k, \Omega) > 0$
so that every self-equilibrated
$F^{ab} \in C^{k+1,\alpha, 2+\beta}(T\Omega^{\odot 2})$
admits a self-equilibrated
extension
$\overline{F}^{ab} \in C^{k+1, \alpha, 2+\beta}
  (T\R^n \odot T\R^n)$
such that
$\overline{F}|_{\R^n \backslash \Omega}$
has compact support
and
$\norm{\overline{F}}_{k+1,\alpha,2+\beta}
  \leq C\norm{F}_{k+1, \alpha, 2+\beta}$.
\end{cor}

\begin{proof}
By Proposition \ref{prescribingricci}
$\widecheck{F}_{ab}$ admits a ``potential''
$h_{ab} \in C^{k+3, \alpha, \beta}(T\Omega^{\odot 2})$
such that $\dot{R}_{ab}[\delta]h = \widecheck{F}$.
Extend $h_{ab}$ to
$\overline{h}_{ab}
  \in C^{k+3, \alpha, \beta}(T\R^n \odot T\R^n)$
and take
$\overline{F}
 =
 (\dot{R}_{ab}[\delta]\widehat{h})\widehat{\hphantom{a}}$.
\end{proof}

Finally we observe that self-equilibration
can be achieved at the cost of a Lie derivative
of the metric.

\begin{prop}[Self-equilibration modulo Lie derivatives]
\label{selfequilibration}
Assume that $\partial \Omega$ is connected
with outward unit normal $N$
and let $\beta' \in \R \backslash \Z$.
There exists $C=C(\Omega, k, \alpha, \beta')>0$
such that
for any
$F_{ab} \in C^{k+1,\alpha,1+\beta'}(T\Omega^{\odot 2})$
there is a vector field $X^a$ on $\Omega$
such that
$\left(
  F_{ab} + X_{a|b} + X_{b|a}
  \right)\widehat{\vphantom{a}}$
is self-equilibrated
and
$\norm{X}_{k+2,\alpha,\beta'}
 \leq
C\norm{\div \widehat{F}}_{k,\alpha,2+\beta'}
 + C \sup \{\abs{\int_{\partial \Omega} F_{ab}Y^aN^b}
     \; : \;
   Y \mbox{ is Killing on } \R^n \}.
$
\end{prop}

\begin{proof}
Since
$\div_a \widehat{L_X \delta} = {X_{a|b}}^{|b}$,
we take $X \in C^{k+2, \alpha, \beta'+1}(T\Omega)$
to be a solution to
${X_{a|b}}^{|b}=-\widehat{F}_{ab}^{\;\;\;|b}$.
We may now assume that $\Omega$ is unbounded 
and $\beta'<n-1$,
since otherwise,
by the divergence theorem
(and the at most linear growth of the Killing fields),
we are already done.
Write $K$ for the Killing operator
taking a vector field $Y$
to $Y_{a|b}+Y_{b|a}$,
and
for any Killing field $Y$ on $\R^n$
write $\overline{Y}$ for the vector field
agreeing with $Y$ on $\partial \Omega$ (including
components normal to $\partial \Omega$)
and having
bounded harmonic Cartesian components on $\Omega$.
For any two Killing fields $Y$ and $Z$ on $\R^n$
consider the pairing
\begin{equation}
  \mathcal{A}(Y, Z)
  :=
  \int_{\partial \Omega}
  \widehat{\left(K\overline{Y}\right)}_{ab}Z^aN^b
  =
  \int_{\Omega}
    \widehat{\left(K\overline{Y}\right)}_{ab}
      \overline{Z}^{a|b}
  =
  \frac{1}{2}\int_{\Omega}
    \widehat{\left(K\overline{Y}\right)}_{ab}
    \left(K\overline{Z}\right)^{ab},
\end{equation}
where the integral at infinity vanishes
because of the decay of
$\overline{Y}$ and $\overline{Z}$.
In particular $\mathcal{A}$ is symmetric.
Next, using the fact that $\widehat{KY}=0$,
we find
\begin{equation}
  \widehat{\left(K\overline{Y}\right)}_{ab}Z^aN^b
  =
  \left(\overline{Y}_{a|b}-Y_{a|b}\right)Z^aN^b,
\end{equation}
and consequently
\begin{equation}
  \mathcal{A}(Y,Z)
  =
  \int_{\Omega} \overline{Y}_{a|b}\overline{Z}^{a|b}
    + \int_{\R^n \backslash \Omega}
     {Y}_{a|b}Z^{a|b},
\end{equation}
so that $\mathcal{A}$ is positive definite.
Thus we may adjust the $X$ chosen above
by harmonic extensions
of restrictions to $\partial \Omega$
of Killing fields
to complete the arrangement of the self-equilibration
condition.
Note that if $\beta' > n-2$,
then
prior to any such adjustments
$\widehat{(F-KX)}_{ab}N^b$ on $\partial \Omega$
is already orthogonal
to the translational Killing fields,
ensuring the estimate in all cases.
\end{proof}

\end{document}